\documentclass[12pt, reqno, a4paper]{amsart}

\usepackage{amssymb}
\usepackage{hyperref}
\usepackage{lscape} %[pdftex]

\usepackage[english]{babel}

\theoremstyle{plain}
\newtheorem{proposition}{Proposition}
\newtheorem*{corollary}{Corollary}
\newtheorem{lemma}{Lemma}
\newtheorem{theorem}{Theorem}
\newtheorem*{conjecture}{Conjecture}

\theoremstyle{remark}
\newtheorem*{remark}{Remark}

\theoremstyle{definition}

\newtheorem{example}{Example}

\DeclareMathOperator{\Id}{Id}

\DeclareMathOperator{\id}{id}

\DeclareMathOperator{\UT}{UT}

\DeclareMathOperator{\supp}{supp}

\DeclareMathOperator{\End}{End}

\DeclareMathOperator{\sign}{sign}

\DeclareMathOperator{\op}{op}

\DeclareMathOperator{\PIexp}{PIexp}

\begin{document}

\title[Semigroup graded algebras]{Semigroup graded algebras and codimension growth of graded polynomial identities}

\author{A.\,S.~Gordienko}
\address{Vrije Universiteit Brussel, Belgium}
\email{alexey.gordienko@vub.ac.be} 


\begin{abstract}
We show that if $T$ is any of four semigroups of two elements that are not groups, there exists
a finite dimensional associative $T$-graded algebra over a field of characteristic $0$
such that the codimensions of its graded polynomial identities have a non-integer exponent of growth.
In particular, we provide an example of a finite dimensional graded-simple semigroup graded algebra
over an algebraically closed field of characteristic $0$ with a non-integer graded PI-exponent,
which is strictly less than the dimension of the algebra.
However, if $T$ is a left or right zero band and the $T$-graded algebra is unital,
or $T$ is a cancellative semigroup, then the $T$-graded algebra satisfies the graded analog of Amitsur's conjecture, i.e. there exists an integer graded PI-exponent. Moreover, in the first case it turns out that
the ordinary and the graded PI-exponents coincide. In addition, we consider related problems on the structure of semigroup graded algebras.
\end{abstract}


\thanks{Supported by Fonds voor Wetenschappelijk Onderzoek~--- Vlaanderen Pegasus Marie Curie post doctoral fellowship (Belgium) and RFBR grant 13-01-00234a (Russia).}

\maketitle

The notion of a semigroup graded algebra is a natural generalization of the notion of a group graded algebra, however the first notion is much less restricting: e.g. if an algebra is the direct sum of its left ideals or if an algebra is the direct sum of a subalgebra and an ideal, this can be expressed in the language of semigroup gradings.

In 2010--2011
E.~Aljadeff,  A.~Giambruno, and D.~La~Mattina~\cite{AljaGia, AljaGiaLa, GiaLa}
 proved that if an associative PI-algebra is graded by a finite group, then there exists an integer exponent of codimensions of its graded polynomial identities, i.e. the graded analog of Amitsur's conjecture holds. In~\cite[Theorem~1]{ASGordienko5} and~\cite[Theorem~3]{ASGordienko9} the author proved the same for finite dimensional associative and Lie algebras graded by any groups. In~\cite{KelarevPI} A.\,V.~Kelarev studied semigroup graded PI-algebras.
 
The next question that naturally arises in this investigation is as to whether the results on codimension growth of graded polynomial identities hold for semigroup graded associative algebras.
 
In the associative case the main properties that we use in order to prove the graded analog of Amitsur's conjecture (Theorem~\ref{TheoremMainTGrAssoc}) are the gradedness (or homogeneity) of the Jacobson radical and the graded version of the Wedderburn~--- Artin theorem.
We consider these properties in Sections~\ref{SectionSemigroupJacobson} and~\ref{SectionSemigroupWedderburn} and obtain the graded analog of Amitsur's conjecture for algebras
graded by cancellative semigroups (Theorem~\ref{TheoremTCancelAmitsur})
and unital algebras graded by left or right zero bands (Theorem~\ref{TheoremTIdemAmitsur}).
In the first case we use Kelarev and Plant's result on gradedness of Jacobson radicals
in algebras graded by cancellative groupoids~\cite[Corollary 4.1]{KelarevBook}.

Until now, there were no examples known of an associative algebra with a non-integer PI-exponent of any kind (graded, Hopf, etc.). In 1999 S.\,P.~Mishchenko and M.\,V.~Zaicev gave an example of an
infinite dimensional Lie algebra with a non-integer PI-exponent~\cite{ZaiMishchFracPI}
(see the proof in~\cite{VerZaiMishch}).
Here we use their ideas to present a finite dimensional semigroup graded associative algebra with a non-integer exponent of codimension growth of graded polynomial identities (Theorems~\ref{TheoremT1GradFractPI}--\ref{TheoremT3GradFractPI}) for each of four semigroups of two elements that are not groups.

The PI-exponent of a finite dimensional graded-simple group graded Lie or associative algebra over an algebraically closed field of characteristic $0$ equals the dimension of the algebra~\cite[Example~12]{ASGordienko5} and~\cite[Theorem~4]{ASGordienko9}. In Theorem~\ref{TheoremT3GradFractPI} we provide an example of a finite dimensional graded-simple semigroup graded algebra
over an algebraically closed field of characteristic $0$ with a non-integer graded PI-exponent,
which is strictly less than the dimension of the algebra.

\section{Semigroups of two elements}

First we describe all the five non-isomorphic semigroups of two elements.

Let $T_1 =\lbrace 0,1\rbrace$ be the multiplicative semigroup of the field $\mathbb Z_2$.

Let $T_2 =\lbrace 0,v\rbrace$ be the semigroup defined by relations $v^2 =0^2= 0\cdot v=v\cdot 0=0$.

Recall that a semigroup $T$ is a \textit{left zero band}
if $t_1 t_2 = t_1$ for every $t_1, t_2 \in T$
and a \textit{right zero band}
if $t_1 t_2 = t_2$ for every $t_1, t_2 \in T$.

Let $T_3$ be the right zero band of two elements.

\begin{proposition}
Let $T$ be a semigroup that consists of two elements. Then $T$ is isomorphic to one of semigroups
from the list $\lbrace T_1, T_2, T_3, T_3^{\,\op}, (\mathbb Z_2,+)\rbrace$ and each two semigroups from this list are non-isomorphic. (Here  $T_3^{\,\op}$ is anti-isomorphic to $T_3$.)
\end{proposition}
\begin{proof}
First, consider the case when $T = \lbrace a, a^2 \rbrace$ for some $a\in T$.
Then if $a^3=a$, we have $a^4=a^2$, $a^2$ is the identity element of $T$, and
$T \cong (\mathbb Z_2,+)$. If $a^3=a^2$, then $a^4=a^3=a^2$, $a^2$ is the zero element of $T$,
and $T \cong T_2$.

Now consider the case when $T \ne \lbrace a, a^2 \rbrace$ for all $a\in T$.
Then $T = \lbrace a, b \rbrace$, $a^2=a$, $b^2=b$. If $ab=ba$, then $T \cong T_1$.
If $ab\ne ba$, then $T\cong T_3$ for $ab = b$, $ba=a$, and $T\cong T_3^{\,\mathrm{op}}$
for $ab = a$, $ba=b$.
\end{proof}

\section{Gradedness of the Jacobson radical}\label{SectionSemigroupJacobson}

Let $T$ be a semigroup. An algebra $A$ is \textit{$T$-graded} if $A=\bigoplus_{t\in T} A^{(t)}$ (direct sum of subspaces) and $A^{(h)}A^{(t)} \subseteq A^{(ht)}$ for all $h,t\in T$. A subspace $V$ of $A$ is \textit{graded}
(or \textit{homogeneous}) if $V=\bigoplus_{t\in T} V \cap A^{(t)}$.

In particular, $T_1$-graded algebras are exactly the algebras with a fixed decomposition
into the direct sum of a two-sided ideal and a subalgebra. If $T$ is a right zero band, then
$T$-graded algebras are algebras with a fixed decomposition into the direct sum of left ideals
indexed by the elements of $T$.

It is known~\cite[Example 4.2]{KelarevBook}, that the Jacobson radical is not necessarily graded.
(See also the survey of positive results in~\cite[Section 4.4]{KelarevBook}.)
Here we provide examples and results related to semigroups of two elements and left and right zero bands.

Denote by $M_k(F)$ is the full $k\times k$ matrix algebra over a field $F$
and $\UT_k(F)$ is the algebra of upper triangular $k\times k$ matrices. In $M_k(F)$
we fix the basis of matrix units $e_{i\ell}$, $1\leqslant i,\ell \leqslant k$.

\begin{example}\label{ExampleT1} 
Let $A = M_k(F)\oplus \UT_k(F)$ (direct sum of ideals) where $F$ is a field, $k\geqslant 2$. Define a $T_1$-grading on $A$ by $A^{(0)}=(M_k(F),0)$, $A^{(1)}=\lbrace (\varphi(a), a) \mid a \in \UT_k(F) \rbrace$ where $\varphi \colon \UT_k(F) \hookrightarrow M_k(F)$ is the natural embedding. Then $$J(A) = \lbrace (0,e_{ij}) \mid 1 \leqslant i < j \leqslant k\rbrace \subset (0,\UT_k(F)),$$ $J(A) \cap A^{(0)} = J(A) \cap A^{(1)}=0$,
and $J(A)$ is not a graded ideal.
\end{example}

\begin{example}\label{ExampleT2}
Let $A = M_k(F)\oplus V$ (direct sum of ideals) where $V\cong M_k(F)$ as a vector space, $k\in \mathbb N$, $V^2 = 0$, and $F$ is a field. Denote by $\varphi \colon V \mathrel{\widetilde{\rightarrow}} M_k(F)$
the corresponding linear isomorphism.
 Define a $T_2$-grading on $A$ by $A^{(0)}=(M_k(F),0)$, $A^{(v)}=\lbrace (\varphi(a), a) \mid a \in V \rbrace$. Then $$J(A) = (0,V),\ J(A) \cap A^{(0)} = J(A) \cap A^{(v)}=0,$$
and $J(A)$ is not a graded ideal.
\end{example}

\begin{example} \label{ExampleT3}
Let $A = M_k(F)\oplus V$ (direct sum of left ideals) where $V\cong M_k(F)$ as a left
$M_k(F)$-module, $k\in \mathbb N$, $V^2=V M_k(F) = 0$, and $F$ is a field. Denote by $\varphi \colon V \mathrel{\widetilde{\rightarrow}} M_k(F)$
the corresponding isomorphism. Define a $T_3$-grading on $A$ by $A^{(e_1)}=(M_k(F),0)$, $$A^{(e_2)}=\lbrace (\varphi(a), a) \mid a \in V \rbrace.$$ Then $$J(A) = (0,V),\ J(A) \cap A^{(e_1)} = J(A) \cap A^{(e_2)}=0,$$
and $J(A)$ is not a graded ideal.
\end{example}

\begin{remark}
One can use the opposite example to show that the Jacobson radical is not necessarily $T_3^{\,\mathrm{op}}$-graded.
\end{remark}

However, if an algebra is unital and $T_3$- or $T_3^{\,\mathrm{op}}$-graded,
then the Jacobson radical is graded. In fact, a more general result holds.

\begin{proposition}\label{PropositionTIdemGradedIdeals}
Let $A$ be a $T$-graded associative algebra with $1$
over a field $F$ for some left or right zero band $T$. Then every ideal of $A$ is graded.
\end{proposition}
\begin{proof}
Consider the case when $t_1 t_2 = t_2$ for every $t_1, t_2 \in T$. (Another case is considered analogously.)
Then all $A^{(t)}$, $t\in T$, are left ideals of $A$ and $1 = \sum_{t\in T}e_t$ for some $e_t \in A^{(t)}$.
Let $I$ be an ideal. Then for every $a\in I$ we have $a=\sum_{t\in T}a e_t$
where $ae_t\in I \cap A^{(t)}$ for every $t\in T$. Hence $I=\bigoplus_{t\in T} I \cap A^{(t)}$
is a graded ideal. 
\end{proof}

\section{Graded analogs of the Wedderburn theorems and $T$-graded simplicity}\label{SectionSemigroupWedderburn}

Now we study whether the graded analogs of the Wedderburn theorems hold for $T$-graded algebras
where $T$ is a semigroup. Recall that a $T$-graded algebra $A$ is a \textit{graded-simple algebra}
if $A^2\ne 0$ and $A$ has no graded ideals other than $A$ and $0$.

\begin{example}\label{ExampleT1Wedderburn}
Let $B=M_k(F)\oplus M_k(F)$
(direct sum of ideals), $k\in \mathbb N$, where $F$ is a field.
Define a $T_1$-grading on $A$ by $A^{(0)}=(M_k(F),0)$, $A^{(1)}=\lbrace (a, a) \mid a \in M_k(F) \rbrace$.
Then $B$ cannot be presented as the direct sum of $T_1$-graded ideals that are $T_1$-graded-simple algebras,
i.e. the $T_1$-graded analog of the Wedderburn~--- Artin theorem does not hold.
\end{example}
\begin{proof} Note that the semisimple algebra $B$ has only four ideals: $0$, $B$, $(0,M_k(F))$,
and $(M_k(F),0)$. Three of them are $T_1$-graded, namely, $0$, $B$,
and $(M_k(F),0)$, and only $(M_k(F),0)$ is a $T_1$-graded-simple algebra.
\end{proof}
\begin{remark}
Since $B^{(0)}$ is always a graded ideal,
every $T_1$-graded-simple algebra $B$ has the trivial grading, i.e. $B=B^{(0)}$.
Therefore, every $T_1$-graded-simple algebra is simple as an ordinary algebra.
\end{remark}

\begin{proposition}\label{PropositionT2Wedderburn}
Let $B$ be a finite dimensional associative $T_2$-graded semisimple algebra over a field $F$.
Then $B=B^{(0)}$, and by the ordinary Wedderburn~--- Artin theorem, $B$ is
the direct sum of $T_2$-graded ideals that are simple algebras (with the trivial grading).
Moreover, every  $T_2$-graded-simple algebra is simple as an ordinary algebra.
In particular, the $T_2$-graded analog of the Wedderburn~--- Artin theorem holds.
\end{proposition}
\begin{proof}
Suppose $B \ne B^{(0)}$.
Note that $B^{(0)}$ is an ideal and, by the ordinary Wedderburn~--- Artin theorem,
$B=B^{(0)} \oplus I$ for some semisimple ideal $I$ of $B$. However, $(B/B^{(0)})^2=0$
since $(B^{(1)})^2 \subseteq B^{(0)}$, and $I \cong B/B^{(0)}$ cannot be semisimple.
Hence $B=B^{(0)}$, the algebra $B$ has the trivial $T_2$-grading, and we can apply to $B$ the ordinary Wedderburn~--- Artin theorem.
\end{proof}

\begin{proposition}\label{PropositionTIdemWedderburn}
Let $B$ be a finite dimensional associative $T$-graded semisimple algebra over a field $F$
for some left or right zero band $T$.
Then $B$ is the direct sum of $T$-graded ideals that are simple algebras.
In particular, the $T$-graded analog of the Wedderburn~--- Artin theorem holds
and every finite dimensional semisimple $T$-graded-simple algebra is simple as an ordinary algebra.
\end{proposition}
\begin{proof} %Note that any subset of $T$ is again a left or right zero band.
%Therefore, reducing $T$ to the support of the grading, we may assume that $T$ is finite.
Consider the case when $T$ is a right zero band. The other case is considered analogously.
By the ordinary Wedderburn~--- Artin theorem, $$B=B_1\oplus B_2 \oplus \ldots \oplus B_s \text{ (direct sum of ideals)}$$ for some simple algebras $B_i$.
% If $I$ is a left ideal in $B$, then $$I = 1_{B_1} I \oplus 1_{B_2} I
%\oplus \ldots \oplus  1_{B_s} I\text{ (direct sum of ideals),}$$
%and $1_{B_i}I = I 1_{B_i} = I\cap B_i$. Hence if $B=B^{(t_1)}\oplus \ldots \oplus B^{(t_s)}$
%is the $T$-grading, then for every $i$, we have $$B_i = 1_{B_i} B= (1_{B_i} B^{(t_1)})\oplus\ldots\oplus (1_{B_i} B^{(e_s)}) = (B_i\cap B^{(t_1)})\oplus\ldots\oplus (B_i\cap B^{(t_s)}).$$
%Therefore, 
By Proposition~\ref{PropositionTIdemGradedIdeals},
each $B_i$ is a $T$-graded ideal. Now the theorem follows.
\end{proof}
However there exist non-semisimple $T_3$-graded-simple algebras. (See Proposition~\ref{PropositionAT3GrSimple} below.)

By Proposition~\ref{PropositionTIdemGradedIdeals}, if a $T$-graded algebra, where $T$ is a left or right zero band, contains unity, the
its Jacobson radical is graded (as well as all the other ideals).
Therefore, one may ask whether the $T$-graded analog of the Wedderburn~--- Mal'cev theorem
holds for such algebras. In fact, the answer is true.

\begin{theorem}\label{TheoremTIdemGradedWeddMalcev}
Let $A$ be a finite dimensional associative $T$-graded algebra with unity over a field $F$ where
$T$ is a left or right zero band and $A/J(A)$ is a separable algebra. (E.g., $F$ is a perfect field.)
Then there exists a graded maximal semisimple subalgebra $B$ such that $A = B \oplus J$ (direct sum of graded spaces) where $J := J(A)$.
\end{theorem}
\begin{proof} Without loss of generality, we may assume that $T$ is a right zero band.
First we consider the case $J^2=0$.

Note that $1_A = \sum_{t\in T} e_t$ for some $e_t\in A^{(t)}$. Moreover, since $1_A e_t=
\sum_{r\in T} e_r e_t$ and $e_r e_t \in A^{(t)}$ for every $t\in T$, we have $e_t^2 = e_t$
and $e_r e_t = 0$ for all $r\ne t$.

Using the ordinary Wedderburn~--- Mal'cev theorem we choose a maximal semisimple subalgebra $B$ such that $A = B \oplus J$ (direct sum of subspaces).
Let $\pi \colon A \twoheadrightarrow A/J$ be the natural projection
which is a graded map since $J$ is graded. Let $\varphi \colon A/J \hookrightarrow A$
be a homomorphic embedding such that $\varphi(A/J)=B$ and $\pi\varphi = \id_{A/J}$.
Note that $\pi(1_A)=\sum_{t\in T} \pi(e_t)$ is the unity of $A/J$.
In addition, $1_A = 1_B = \varphi\pi(1_A)$.

Let $T=\lbrace t_1, \ldots, t_s\rbrace$.
If $\varphi\pi(e_{t_i})=e_{t_i}$ for all $1\leqslant i\leqslant s$,
then $B{e_{t_i}} \subseteq B$ for all $1\leqslant i\leqslant s$ and $B=\bigoplus_{t\in T} B{e_t}$ is a graded subalgebra and the theorem is proved.

Suppose $\varphi\pi(e_{t_i}) \ne e_{t_i}$ for at least one $1\leqslant i\leqslant s$.
Choose $0\leqslant k \leqslant s-1$ such that $\varphi\pi(e_{t_i})=e_{t_i}$
for all $1\leqslant i \leqslant k$
and $\varphi\pi(e_{t_{k+1}}) \ne e_{t_{k+1}}$.
 Note that $\pi(\varphi\pi(e_{t_{k+1}})-e_{t_{k+1}})=0$ and $\varphi\pi(e_{t_{k+1}}) = e_{t_{k+1}}+j$ for some $j\in J$. In addition, $j e_{t_i} = (\varphi\pi(e_{t_{k+1}})-e_{t_{k+1}})e_{t_i}
 = \varphi\pi(e_{t_{k+1}} e_{t_i})-e_{t_{k+1}} e_{t_i} = 0$ for all $1\leqslant i\leqslant k$.
 Analogously, $e_{t_i}j=0$ for all $1\leqslant i\leqslant k$.
 Moreover, since $(e_{t_{k+1}} +j)^2=e_{t_{k+1}} +j$, we have
 $j=e_{t_{k+1}}j + j e_{t_{k+1}}$ and $e_{t_{k+1}} j e_{t_{k+1}} = 0$.
 
 Let $\tilde \varphi \colon A/J \hookrightarrow A$ be the homomorphic embedding defined by \begin{equation*}\begin{split}\tilde\varphi(a) = (1_A + e_{t_{k+1}}j - je_{t_{k+1}})\varphi(a)(1_A + e_{t_{k+1}}j - je_{t_{k+1}})^{-1}
 =\\(1_A + e_{t_{k+1}}j - je_{t_{k+1}})\varphi(a)(1_A - e_{t_{k+1}}j + je_{t_{k+1}}).\end{split}\end{equation*}
 Note that $\pi\tilde\varphi=\id_{A/J}$ and $$\tilde \varphi \pi(e_{t_i})=(1_A + e_{t_{k+1}}j - je_{t_{k+1}})e_{t_i}(1_A - e_{t_{k+1}}j + je_{t_{k+1}})=e_{t_i} \text{ for all } 1\leqslant i \leqslant k.$$
 Moreover \begin{equation*}\begin{split}\tilde \varphi \pi(e_{t_{k+1}})=
 (1_A + e_{t_{k+1}}j - je_{t_{k+1}})\varphi \pi(e_{t_{k+1}})(1_A - e_{t_{k+1}}j + je_{t_{k+1}})
 =\\ 
 (1_A  + e_{t_{k+1}}j - je_{t_{k+1}})(e_{t_{k+1}}+j)(1_A  - e_{t_{k+1}}j + je_{t_{k+1}})=\\
 (e_{t_{k+1}}+j  - je_{t_{k+1}})(1_A  - e_{t_{k+1}}j + je_{t_{k+1}})=\\
e_{t_{k+1}}+j  - je_{t_{k+1}} - e_{t_{k+1}}j =
 e_{t_{k+1}}.\end{split}\end{equation*}
 Therefore $\tilde B = \tilde\varphi(A/J)$ is a maximal semisimple subalgebra
 such that $A = \tilde B \oplus J$ (direct sum of subspaces)
 and $\tilde \varphi \pi(e_{t_i}) = e_{t_i}$ for all $1\leqslant i \leqslant k+1$.
 Thus, using the induction argument, we may assume that 
 $e_t=\tilde \varphi \pi(e_t) \in \tilde B$ for all $t\in T$.
Hence $\tilde B=\bigoplus_{t\in T} \tilde B{e_t}$ is a graded subalgebra of $A$.

We have proved the theorem for the case $J^2=0$. 
The general case is proved by induction on $\dim A$.
Suppose $J^2\ne 0$. Then $A/J^2 = B_0 \oplus J/J^2$ (direct sum of graded subspaces)
for some graded maximal semisimple subalgebra $B_0$ of $A/J^2$.  Note that $1_{A/J^2} \in B_0$. 
Consider the preimage $B_1$ of $B_0$ in $A$ under the natural map $\pi_1 \colon A \twoheadrightarrow A/J^2$. Then $1_A \in B_1$.
Since $B_0 \cong A/J$ is semisimple, $J(B_1)=J^2$.
Moreover $\dim B_1 < \dim A$ and, by the induction assumption, we have $B_1 = B \oplus J^2$
(direct sum of graded subspaces) for some graded maximal semisimple subalgebra $B$ in $A$.
Hence $A = B \oplus J$ (direct sum of graded subspaces) and the theorem is proved.
\end{proof}

Recall that a semigroup $T$ is \textit{cancellative} if for every $a,b,c \in T$ each of the conditions $ac = bc$ and $ca=cb$ implies $a=b$.

\begin{proposition}\label{PropositionTCancelWedderburn}
Let $B$ be a finite dimensional associative $T$-graded semisimple algebra over a field $F$
for some cancellative semigroup $T$.
Then $B$ is the direct sum of $T$-graded ideals that are $T$-graded-simple algebras.
In particular, the $T$-graded analog of the Wedderburn~--- Artin theorem holds.\end{proposition}
\begin{proof}
By the ordinary Wedderburn~--- Artin theorem, $B = B_1 \oplus \ldots \oplus B_s$ (direct sum of ideals)
for some simple algebras $B_i$.
Let $I$ be a minimal $T$-graded ideal in $B$.
Then $I=B_{i_1}\oplus\ldots \oplus B_{i_k}$ for some $i_1, \ldots, i_k$.
Define $N:=\bigoplus_{i \in \lbrace 1,\ldots, s\rbrace \backslash \lbrace i_1, \ldots, i_k\rbrace} B_i$.
Since all $B_i$ are semisimple, we have $N=\lbrace b\in B \mid ba =0 \text{ for all } a\in I\rbrace$.
Since $T$ is cancellative and $I$ is graded, the ideal $N$ is graded too.
Hence $B = I \oplus N$ (direct sum of graded ideals) where $I$ is a $T$-graded-simple algebra. Applying to $N$ the inductive argument, we get the proposition.
\end{proof}

\section{Graded polynomial identities, their codimensions and cocharacters}

Let $T$ be a semigroup and let $F$ be a field. Denote by $F\langle X^{T\text{-}\mathrm{gr}} \rangle $ the free $T$-graded associative  algebra over $F$ on the countable set $$X^{T\text{-}\mathrm{gr}}:=\bigcup_{t \in T}X^{(t)},$$ $X^{(t)} = \{ x^{(t)}_1,
x^{(t)}_2, \ldots \}$,  i.e. the algebra of polynomials
 in non-commuting variables from $X^{T\text{-}\mathrm{gr}}$.
  The indeterminates from $X^{(t)}$ are said to be homogeneous of degree
$t$. The $T$-degree of a monomial $x^{(t_1)}_{i_1} \dots x^{(t_t)}_{i_s} \in F\langle
 X^{T\text{-}\mathrm{gr}} \rangle $ is defined to
be $t_1 t_2 \dots t_s$, as opposed to its total degree, which is defined to be $s$. Denote by
$F\langle
 X^{T\text{-}\mathrm{gr}} \rangle^{(t)}$ the subspace of the algebra $F\langle
 X^{T\text{-}\mathrm{gr}} \rangle$ spanned
 by all the monomials having
$T$-degree $t$. Notice that $$F\langle
 X^{T\text{-}\mathrm{gr}} \rangle^{(t)} F\langle
 X^{T\text{-}\mathrm{gr}} \rangle^{(h)} \subseteq F\langle
 X^{T\text{-}\mathrm{gr}} \rangle^{(th)},$$ for every $t, h \in T$. It follows that
$$F\langle
 X^{T\text{-}\mathrm{gr}} \rangle =\bigoplus_{t\in T} F\langle
 X^{T\text{-}\mathrm{gr}} \rangle^{(t)}$$ is a $T$-grading.
  Let $f=f(x^{(t_1)}_{i_1}, \dots, x^{(t_s)}_{i_s}) \in F\langle
 X^{T\text{-}\mathrm{gr}} \rangle$.
We say that $f$ is
a \textit{graded polynomial identity} of
 a $T$-graded algebra $A=\bigoplus_{t\in T}
A^{(t)}$
and write $f\equiv 0$
if $f(a^{(t_1)}_{i_1}, \dots, a^{(t_s)}_{i_s})=0$
for all $a^{(t_j)}_{i_j} \in A^{(t_j)}$, $1 \leqslant j \leqslant s$.
  The set $\Id^{T\text{-}\mathrm{gr}}(A)$ of graded polynomial identities of
   $A$ is
a graded ideal of $F\langle
 X^{T\text{-}\mathrm{gr}} \rangle$.
%The case of ordinary polynomial identities is included
%for the trivial semigroup $G=\lbrace e \rbrace$.

\begin{example}\label{ExampleIdGr}
 Let $T=(\mathbb Z_2,+) = \lbrace \bar 0, \bar 1 \rbrace$,
$M_2(F)=M_2(F)^{(\bar 0)}\oplus M_2(F)^{(\bar 1)}$
where $M_2(F)^{(\bar 0)}=\left(
\begin{array}{cc}
F & 0 \\
0 & F
\end{array}
 \right)$ and $M_2(F)^{(\bar 1)}=\left(
\begin{array}{cc}
0 & F \\
F & 0
\end{array}
 \right)$. Then  $x^{(\bar 0)} y^{(\bar 0)} - y^{(\bar 0)} x^{(\bar 0)}
\in \Id^{T\text{-}\mathrm{gr}}(M_2(F))$.
\end{example}

Let
$P^{T\text{-}\mathrm{gr}}_n := \langle x^{(t_1)}_{\sigma(1)}
x^{(t_2)}_{\sigma(2)}\ldots x^{(t_n)}_{\sigma(n)}
\mid t_i \in T, \sigma\in S_n \rangle_F \subset F \langle X^{T\text{-}\mathrm{gr}} \rangle$, $n \in \mathbb N$.
Then the number $$c^{T\text{-}\mathrm{gr}}_n(A):=\dim\left(\frac{P^{T\text{-}\mathrm{gr}}_n}{P^{T\text{-}\mathrm{gr}}_n \cap \Id^{T\text{-}\mathrm{gr}}(A)}\right)$$
is called the $n$th \textit{codimension of graded polynomial identities}
or the $n$th \textit{graded codimension} of $A$.

The analog of Amitsur's conjecture for graded codimensions can be formulated
as follows.

\begin{conjecture} There exists
 $\PIexp^{T\text{-}\mathrm{gr}}(A):=\lim\limits_{n\to\infty} \sqrt[n]{c^{T\text{-}\mathrm{gr}}_n(A)} \in \mathbb Z_+$.
\end{conjecture}

If $T$ is the trivial (semi)group of one element, we get the notion of ordinary polynomial identities, ordinary codimensions $c_n(A)$, and the ordinary PI-exponent $\PIexp(A)$.

As we shall see in Theorems~\ref{TheoremT1GradFractPI}--\ref{TheoremT3GradFractPI} below, the analog of Amitsur's conjecture fails for all semigroups $T$ of two elements that are not groups.

However, in Theorem~\ref{TheoremMainTGrAssoc} below we provide sufficient conditions for a graded algebra to satisfy the analog of Amitsur's conjecture.
As a consequence, we prove that if $T$ is a cancellative semigroup or
$T$ is a left or right zero band, and a finite dimensional $T$-graded algebra $A$ contains $1$, then $A$ satisfies the graded analog of Amitsur's conjecture (Theorems~\ref{TheoremTCancelAmitsur} and~\ref{TheoremTIdemAmitsur}).

\section{Polynomial $H$-identities and their codimensions}  

In our case, instead of working with graded codimensions directly, it is more convenient to replace the grading with the corresponding dual structure and study the asymptotic behaviour of polynomial $H$-identities.

  Let $H$ be an arbitrary associative algebra with $1$ over a field $F$.
We say that an associative algebra $A$ is an algebra with a \textit{generalized $H$-action}
if $A$ is endowed with a homomorphism $H \to \End_F(A)$
and for every $h \in H$ 
there exist $k\in \mathbb N$ and $h'_i, h''_i, h'''_i, h''''_i \in H$, $1\leqslant i \leqslant k$,
such that 
\begin{equation}\label{EqGenHAction}
h(ab)=\sum_{i=1}^k\bigl((h'_i a)(h''_i b) + (h'''_i b)(h''''_i a)\bigr) \text{ for all } a,b \in A.
\end{equation}

\begin{remark}
We use the term ``generalized $H$-action'' in order to distinguish from the case when an algebra is
an $H$-module algebra for some Hopf algebra $H$ which is a particular case of the generalized $H$-action.
\end{remark}

Let $F \langle X \rangle$ be the free associative algebra without $1$
   on the set $X := \lbrace x_1, x_2, x_3, \ldots \rbrace$.
  Then $F \langle X \rangle = \bigoplus_{n=1}^\infty F \langle X \rangle^{(n)}$
  where $F \langle X \rangle^{(n)}$ is the linear span of all monomials of total degree $n$.
%   Let $H$ be an arbitrary associative algebra with $1$ over $F$. 
   Consider the algebra $$F \langle X | H\rangle
   :=  \bigoplus_{n=1}^\infty H^{{}\otimes n} \otimes F \langle X \rangle^{(n)}$$
   with the multiplication $(u_1 \otimes w_1)(u_2 \otimes w_2):=(u_1 \otimes u_2) \otimes w_1w_2$
   for all $u_1 \in  H^{{}\otimes j}$, $u_2 \in  H^{{}\otimes k}$,
   $w_1 \in F \langle X \rangle^{(j)}$, $w_2 \in F \langle X \rangle^{(k)}$.
We use the notation $$x^{h_1}_{i_1}
x^{h_2}_{i_2}\ldots x^{h_n}_{i_n} := (h_1 \otimes h_2 \otimes \ldots \otimes h_n) \otimes x_{i_1}
x_{i_2}\ldots x_{i_n}.$$ Here $h_1 \otimes h_2 \otimes \ldots \otimes h_n \in H^{{}\otimes n}$,
$x_{i_1} x_{i_2}\ldots x_{i_n} \in F \langle X \rangle^{(n)}$. 

Note that if $(\gamma_\beta)_{\beta \in \Lambda}$ is a basis in $H$, 
then $F\langle X | H \rangle$ is isomorphic to the free associative algebra over $F$ with free formal  generators $x_i^{\gamma_\beta}$, $\beta \in \Lambda$, $i \in \mathbb N$.
 We refer to the elements
 of $F\langle X | H \rangle$ as \textit{associative $H$-polynomials}.
Note that here we do not consider any $H$-action on $F \langle X | H \rangle$.

Let $A$ be an associative algebra with a generalized $H$-action.
Any map $\psi \colon X \to A$ has the unique homomorphic extension $\bar\psi
\colon F \langle X | H \rangle \to A$ such that $\bar\psi(x_i^h)=h\psi(x_i)$
for all $i \in \mathbb N$ and $h \in H$.
 An $H$-polynomial
 $f \in F\langle X | H \rangle$
 is an \textit{$H$-identity} of $A$ if $\bar\psi(f)=0$
for all maps $\psi \colon X \to A$. In other words, $f(x_1, x_2, \ldots, x_n)$
 is an $H$-identity of $A$
if and only if $f(a_1, a_2, \ldots, a_n)=0$ for any $a_i \in A$.
 In this case we write $f \equiv 0$.
The set $\Id^{H}(A)$ of all $H$-identities
of $A$ is an ideal of $F\langle X | H \rangle$.

We denote by $P^H_n$ the space of all multilinear $H$-polynomials
in $x_1, \ldots, x_n$, $n\in\mathbb N$, i.e.
$$P^{H}_n = \langle x^{h_1}_{\sigma(1)}
x^{h_2}_{\sigma(2)}\ldots x^{h_n}_{\sigma(n)}
\mid h_i \in H, \sigma\in S_n \rangle_F \subset F \langle X | H \rangle.$$
Then the number $c^H_n(A):=\dim\left(\frac{P^H_n}{P^H_n \cap \Id^H(A)}\right)$
is called the $n$th \textit{codimension of polynomial $H$-identities}
or the $n$th \textit{$H$-codimension} of $A$.

One of the main tools in the investigation of polynomial
identities is provided by the representation theory of symmetric groups.
 The symmetric group $S_n$  acts
 on the space $\frac {P^H_n}{P^H_{n}
  \cap \Id^H(A)}$
  by permuting the variables.
  Irreducible $FS_n$-modules are described by partitions
  $\lambda=(\lambda_1, \ldots, \lambda_s)\vdash n$ and their
  Young diagrams $D_\lambda$.
   The character $\chi^H_n(A)$ of the
  $FS_n$-module $\frac {P^H_n}{P^H_n
   \cap \Id^H(A)}$ is
   called the $n$th
  \textit{cocharacter} of polynomial $H$-identities of $A$.
  We can rewrite it as
  a sum $$\chi^H_n(A)=\sum_{\lambda \vdash n}
   m(A, H, \lambda)\chi(\lambda)$$ of
  irreducible characters $\chi(\lambda)$.
Let  $e_{T_{\lambda}}=a_{T_{\lambda}} b_{T_{\lambda}}$
and
$e^{*}_{T_{\lambda}}=b_{T_{\lambda}} a_{T_{\lambda}}$
where
$a_{T_{\lambda}} = \sum_{\pi \in R_{T_\lambda}} \pi$
and
$b_{T_{\lambda}} = \sum_{\sigma \in C_{T_\lambda}}
 (\sign \sigma) \sigma$,
be Young symmetrizers corresponding to a Young tableau~$T_\lambda$.
Then $M(\lambda) = FS_n e_{T_\lambda} \cong FS_n e^{*}_{T_\lambda}$
is an irreducible $FS_n$-module corresponding to
 a partition~$\lambda \vdash n$.
  We refer the reader to~\cite{Bahturin, DrenKurs, ZaiGia}
   for an account
  of $S_n$-representations and their applications to polynomial
  identities.

\section{Generalized $(FT)^*$-action on $T$-graded algebras}

In this section we show that every finite dimensional semigroup graded algebra is an algebra with a generalized $H$-action for a suitable associative algebra $H$.

For an arbitrary semigroup $T$ one can consider the \textit{semigroup algebra} $FT$ over a field $F$ which is the vector space with the formal basis $(t)_{t\in T}$ and the multiplication induced by the one in $T$.

Consider the vector space $(FT)^*$ dual to $FT$. Then $(FT)^*$ is an algebra with the multiplication defined
by $(hw)(t)=h(t)w(t)$ for $h,w \in (FT)^*$ and $t\in T$. The identity element 
is defined by $1_{(FT)^*}(t)=1$ for all $t\in T$. In other words, $(FT)^*$ is the algebra dual to the coalgebra $FT$.

Let $\Gamma \colon A=\bigoplus_{t\in T} A^{(t)}$ be a grading on an algebra $A$. We have the following natural $(FT)^*$-action on $A$: $h a^{(t)}:=h(t)a^{(t)}$ for all $h \in (FT)^*$, $a^{(t)}\in A^{(t)}$
and $t\in T$.

\begin{remark}
If $T$ is a finite group, then $A$ is an $FT$-comodule algebra for the Hopf algebra $FT$
and an $(FT)^*$-module algebra for the Hopf algebra $(FT)^*$.
\end{remark}

For every $t\in T$ define $h_t \in (FT)^*$ by $h_t(g)=\left\lbrace \begin{array}{lll}
0 & \text{if} & g\ne t, \\
1 & \text{if} & g = t
\end{array}\right.$ for $g\in T$.

If $A$ is finite dimensional,
the set $\supp \Gamma := \lbrace t\in T
\mid A^{(t)}\ne 0 \rbrace$ is finite and $$h_t(ab)=\sum\limits_{\substack{g, w \in \supp \Gamma,\\
gw=t}} h_g(a)h_w(b)\text{ for all }a,b\in A.$$

Note that $ha = \sum_{t\in \supp \Gamma}h(t)h_t a$ for all $a\in A$
and \begin{equation}\label{EqIdentityHFiniteSupp}x^h - \sum_{t\in \supp \Gamma}h(t) x^{h_t} \in \Id^{(FT)^*}(A)
\end{equation}
for all $h\in (FT)^*$.
 By linearity, we get~(\ref{EqGenHAction}). Therefore,  $A$ is an algebra with a generalized $(FT)^*$-action.

\begin{lemma}\label{LemmaCnGrCnGenH}
Let $A$ be a finite dimensional algebra over a field $F$ graded by a semigroup $T$.
Then $c_n^{T\text{-}\mathrm{gr}}(A)=c_n^{(FT)^*}(A)$ for all $n\in \mathbb N$.
\end{lemma}
\begin{proof} Denote the grading $ A=\bigoplus_{t\in T} A^{(t)}$ by $\Gamma$.

Let $$\xi \colon F\langle X \mid (FT)^* \rangle \to F\langle X^{T\text{-}\mathrm{gr}} \rangle$$ be the homomorphism  of algebras defined by $\xi(x_i^h) = \sum\limits_{t\in\supp \Gamma} h(t)x^{(t)}_i$, $i\in\mathbb N$, $h\in (FT)^*$. Suppose $f\in \Id^{(FT)^*}(A)$. Consider an arbitrary graded homomorphism $\psi \colon  
F\langle X^{T\text{-}\mathrm{gr}} \rangle \to A$. Then the homomorphism  of algebras $\psi\xi \colon F\langle X \mid (FT)^* \rangle
\to A$ satisfies the condition $$\psi\xi(x_i^h)=\sum\limits_{t\in\supp \Gamma} h(t)\psi\left(x^{(t)}_i\right)=
h\left(\sum\limits_{t\in\supp \Gamma} \psi\left(x^{(t)}_i\right)\right)=h\,\psi\xi(x_i).$$ Thus $\psi\xi(f) =0$ and $\xi(f)\in \Id^{T\text{-}\mathrm{gr}}(A)$. Hence $\xi\left(\Id^{(FT)^*}(A)\right)\subseteq \Id^{T\text{-}\mathrm{gr}}(A)$.
Denote by $$\tilde \xi \colon F\langle X \mid (FT)^* \rangle/\Id^{(FT)^*}(A) \to F\langle X^{T\text{-}\mathrm{gr}} \rangle/\Id^{T\text{-}\mathrm{gr}}(A)$$ the homomorphism induced by $\xi$.

Let $$\eta \colon F\langle X^{T\text{-}\mathrm{gr}} \rangle \to F\langle X \mid (FT)^* \rangle$$
be the homomorphism defined by $\eta\left(x^{(t)}_i\right) = x^{h_t}_i$ for all $i\in \mathbb N$
and $t\in T$. Consider an arbitrary graded polynomial identity $f\in F\langle X^{T\text{-}\mathrm{gr}} \rangle$.
Let $\psi \colon  F\langle X \mid (FT)^* \rangle \to A$ be a homomorphism satisfying the condition
$\psi(x_i^h)=h\psi(x_i)$ for every $i\in\mathbb N$ and $h\in (FT)^*$.
Then for any $i\in\mathbb N$ and $g, t \in T$ we have
$$h_g \psi\eta\left(x^{(t)}_i\right) = h_g\psi(x^{h_t}_i)=h_g h_t \psi(x_i)
=\left\lbrace \begin{array}{lll} 0 & \text{ if } & g\ne t,\\
                              \psi\eta\left(x^{(t)}_i\right) & \text{ if } & g=t. \end{array}\right.$$
 Thus $\psi\eta\left(x^{(t)}_i\right) \in A^{(t)}$ and $\psi\eta$ is a graded homomorphism.
 Therefore, $\psi\eta(f)=0$ and $\eta(\Id^{T\text{-}\mathrm{gr}}(A)) \subseteq \Id^{(FT)^*}(A)$.
Denote by $\tilde\eta \colon  F\langle X^{T\text{-}\mathrm{gr}} \rangle/\Id^{T\text{-}\mathrm{gr}}(A) \to
F\langle X \mid (FT)^* \rangle/\Id^{(FT)^*}(A)$ the induced homomorphism.

Now we use the notation $\bar f = f + \Id^{(FT)^*}(A) \in F\langle X \mid (FT)^* \rangle/\Id^{(FT)^*}(A)$ for $f\in
F\langle X \mid (FT)^* \rangle$ and  $\bar f = f + \Id^{T\text{-}\mathrm{gr}}(A) \in F\langle X^{T\text{-}\mathrm{gr}} \rangle/\Id^{T\text{-}\mathrm{gr}}(A)$ for $f\in F\langle X^{T\text{-}\mathrm{gr}} \rangle$.
We have $$\tilde\eta\tilde\xi\left(\bar x^h_i\right)=\tilde\eta\left(
\sum\limits_{t\in\supp \Gamma} h(t) \bar x^{(t)}_i\right)
=\sum\limits_{t\in\supp \Gamma} h(t) \bar x^{h_t}_i = \bar x^h_i$$
for every $h\in (FT)^*$ and $i\in\mathbb N$. (Here we use~(\ref{EqIdentityHFiniteSupp}).)
Thus $\tilde\eta\tilde\xi=\id_{F\langle X \mid (FT)^* \rangle/\Id^{(FT)^*}(A)}$.
Moreover $\tilde\xi\tilde\eta\left(\bar x^{(t)}_i\right)=
\tilde\xi\left(\bar x^{h_t}_i\right)=\bar x^{(t)}_i$ for every $t\in T$ and $i\in \mathbb N$.
Therefore, $\tilde\xi\tilde\eta=\id_{F\langle X^{T\text{-}\mathrm{gr}} \rangle/\Id^{T\text{-}\mathrm{gr}}(A)}$
and $F\langle X^{T\text{-}\mathrm{gr}} \rangle/\Id^{T\text{-}\mathrm{gr}}(A) \cong F\langle X \mid (FT)^* \rangle/\Id^{(FT)^*}(A)$
as algebras. The restriction of $\tilde\xi$ provides the isomorphism of
$\frac{P^{(FT)^*}_n}{P^{(FT)^*}_n \cap \Id^{(FT)^*}(A)}$ and $\frac{P^{T\text{-}\mathrm{gr}}_n}{P^{T\text{-}\mathrm{gr}}_n\cap \Id^{T\text{-}\mathrm{gr}}(A)}$.  Hence $$c^{(FT)^*}_n(A)=\dim \frac{P^{(FT)^*}_n}{P^{(FT)^*}_n \cap \Id^{(FT)^*}(A)}
= \dim\frac{P^{T\text{-}\mathrm{gr}}_n}{P^{T\text{-}\mathrm{gr}}_n\cap \Id^{T\text{-}\mathrm{gr}}(A)}=c^{T\text{-}\mathrm{gr}}_n(A).$$
\end{proof}

Now we can provide a sufficient condition for a graded algebra to satisfy the graded analog of Amitsur's conjecture.

\begin{theorem}\label{TheoremMainTGrAssoc}
Let $A$ be a finite dimensional non-nilpotent $T$-graded associative algebra
over an algebraically closed field $F$ of characteristic $0$ for some semigroup $T$.
Suppose that the Jacobson radical $J:=J(A)$ is a graded ideal.
Let $$A/J = B_1 \oplus \ldots \oplus B_q \text{ (direct sum of graded ideals)}$$
where $B_i$ are graded-simple algebras and let $\varkappa \colon A/J \to A$
be any homomorphism of algebras (not necessarily graded) such that $\pi\varkappa = \id_{A/J}$ where $\pi \colon A \to A/J$ is the natural projection.
 Then
there exist constants $C_1, C_2 > 0$, $r_1, r_2 \in \mathbb R$
such that $C_1 n^{r_1} d^n \leqslant c^{T\text{-}\mathrm{gr}}_n(A) \leqslant C_2 n^{r_2} d^n$
for all $n \in \mathbb N$
where $$d= \max\dim\left( B_{i_1}\oplus B_{i_2} \oplus \ldots \oplus B_{i_r}
 \mathbin{\Bigl|}  r \geqslant 1,\right.$$ \begin{equation*}\left. ((FT)^*\varkappa(B_{i_1}))A^+ \,((FT)^*\varkappa(B_{i_2})) A^+ \ldots ((FT)^*\varkappa(B_{i_{r-1}})) A^+\,((FT)^*\varkappa(B_{i_r}))\ne 0\right)\end{equation*} and
 $A^+:=A+F\cdot 1$.
\end{theorem}
\begin{proof}
The theorem is an immediate consequence of Lemma~\ref{LemmaCnGrCnGenH}
and~\cite[Theorem~1]{ASGordienko8}.
\end{proof}
\begin{remark}
The existence of the map $\varkappa$ follows from the ordinary Wedderburn~--- Mal'cev theorem.
\end{remark}
\begin{remark}
If $A$ is nilpotent, i.e. $x_1 \ldots x_p \equiv 0$ for some $p\in\mathbb N$, then  $P^{T\text{-}\mathrm{gr}}_n \subseteq \Id^{T\text{-}\mathrm{gr}}(A)$ and $c^{T\text{-}\mathrm{gr}}_n(A)=0$ for all $n \geqslant p$.
\end{remark}
\begin{corollary}
The above analog of Amitsur's conjecture holds for such codimensions.
\end{corollary}

\section{Partitions restricted to convex polytopes}

Here we apply ideas from~\cite{VerZaiMishch} and prove auxiliary results that we use in the construction
of algebras with non-integer graded PI-exponents.

In this section we show that if all the partitions $\lambda\vdash n$ that correspond to
irreducible $FS_n$-modules with nonzero multiplicities $m(A,H,\lambda)$ belong
to a convex polyhedron, then $\mathop{\overline\lim}_{n\to\infty}\sqrt[n]{c_n^{H}(A)}$ is bounded
by the maximum of a particular function $\Phi$ on the ``continuous'' version of the polyhedron.

Fix $q\in\mathbb N$. Let $\Phi(\alpha_1, \ldots, \alpha_q)=\frac{1}{\alpha_1^{\alpha_1} \ldots \alpha_q^{\alpha_q}}$. Define $0^0 := 1$. Then $\Phi$ is continuous on the segment
$\lbrace (\alpha_1, \ldots, \alpha_q) \mid \alpha_i \geqslant 0\rbrace$.

Suppose we have some numbers $\gamma_{ij} \in \mathbb R$ for $1\leqslant i \leqslant m$, $0\leqslant j \leqslant q$ and $\theta_k \in \mathbb Z_+$ for $q< k \leqslant r$ where $m,r\in \mathbb Z_+$, $r\geqslant q$. Define
$$\Omega = \left\lbrace
(\alpha_1, \ldots, \alpha_q)\in \mathbb R^q \mathrel{\Bigl|}\sum_{i=1}^q \alpha_i=1,\ \alpha_1\geqslant\alpha_2\geqslant \ldots \geqslant \alpha_q\geqslant 0,
\ \sum_{j=1}^q \gamma_{ij}\alpha_j \geqslant 0\text{ for } 1\leqslant i \leqslant m\right\rbrace.$$

For every $n\in \mathbb N$ we define
$$\Omega_n = \left\lbrace
\lambda \vdash n \mathrel{\Bigl|} \sum_{j=1}^q \gamma_{ij}\lambda_j+\gamma_{i0} \geqslant 0 \text{ for } 1\leqslant i \leqslant m,\ \lambda_i \leqslant \theta_i \text{ for } q < i \leqslant r,\ 
\lambda_{r+1}=0 \right\rbrace.$$

We treat $\Omega$ and $\Omega_n$ as the ``continuous'' and the ``discrete'' version
of the same polyhedron. 

Denote by $d$ the maximum of $\Phi$ on the compact set $\Omega$. (We assume $\Omega$ to be non-empty.)

\begin{lemma}\label{LemmaTExampleUpperFd}
Let $A$ be an algebra with a generalized $H$-action where $H$ is an associative algebra with unity
over a field $F$ of characteristic $0$. Suppose $m(A, H, \lambda)=0$ for all $\lambda\vdash n$, $\lambda 
\notin \Omega_n$, $n\in\mathbb N$. Then 
$\mathop{\overline\lim}_{n\to\infty}\sqrt[n]{c_n^{H}(A)}
\leqslant d$.
\end{lemma}
\begin{proof}
Let $\lambda \vdash n$ such that $m(A,H,\lambda)\ne 0$.
By the hook formula, $\dim M(\lambda)=\frac{n!}{\prod_{i,j} h_{ij}}$
where $h_{ij}$ is the length of the hook with the edge in $(i,j)$
in the Young diagram $D_\lambda$. Hence
$\dim M(\lambda) \leqslant \frac{n!}{\lambda_1! \ldots \lambda_r!}$.
Note that $\left(x^x\right)'=\left(e^{x\ln x}\right)'=
(\ln x+1)e^{x\ln x}$ and $x^x$ is decreasing for $x\leqslant \frac{1}{e}$.
By the Stirling formula, for all sufficiently large $n$ we have \begin{equation}\begin{split}\label{EqMlambdaUpperFd}\dim M(\lambda) \leqslant \frac{C_1 n^{r_1}
\left(\frac{n}{e}\right)^n}{\left(\frac{\lambda_1}{e}\right)^{\lambda_1}\ldots
\left(\frac{\lambda_r}{e}\right)^{\lambda_r}}=C_1 n^{r_1}\left(\frac{1}
{\left(\frac{\lambda_1}{n}\right)^{\frac{\lambda_1}{n}}\ldots
\left(\frac{\lambda_r}{n}\right)^{\frac{\lambda_r}{n}}}\right)^n
\leqslant \\ C_1 n^{r_1} \left(\Phi\left(\frac{\lambda_1}{n}, \ldots, \frac{\lambda_q}{n}\right)\right)^n
\frac{n^{\theta_{q+1}+\ldots +\theta_r}}{\theta_{q+1}^{\theta_{q+1}} \ldots \theta_r^{\theta_r}}=C_2 n^{r_2} \left(\Phi\left(\frac{\lambda_1}{n}, \ldots, \frac{\lambda_q}{n}\right)\right)^n\end{split}\end{equation}
for some $C_1, C_2 > 0$ and $r_1, r_2 \in\mathbb R$ that do not depend on $\lambda_i$.

 Let $\varepsilon > 0$. Since $\Phi$ is continuous, there exists $\delta > 0$
such that for every $x$ from the domain of $\Phi$ such that the distance between
$x$ and $\Omega$ is less than $\delta$, we have $\Phi(x) < d + \varepsilon$.

Therefore, by~(\ref{EqMlambdaUpperFd}), there exists $n_0\in\mathbb N$ such that
for all $n \geqslant n_0$ and $\lambda\vdash n$ such that $m(A,H,\lambda)\ne 0$
 we have $ \dim M(\lambda) \leqslant C_2 n^{r_2} (d+\varepsilon)^n$.

 By~\cite[Theorem~5]{ASGordienko8},
there exist $C_3 > 0$, $r_3\in\mathbb Z_+$ such that
 $$\sum_{\lambda \vdash n} m(A,H,\lambda)
\leqslant C_3 n^{r_3}\text{ for all }n \in \mathbb N.$$

Hence $$ c^{H}_n(A) = \sum_{\lambda \vdash n} m(A,H,\lambda) \dim M(\lambda)
 \leqslant  C_2 C_3 n^{r_2+r_3} (d+\varepsilon)^n$$
 and $\mathop{\overline\lim}_{n\to\infty}\sqrt[n]{c_n^H(A)}
\leqslant d+\varepsilon$. Since $\varepsilon > 0$ is arbitrary, we get the lemma.
\end{proof}

\begin{lemma}\label{LemmaMaxTExample}
Let $q \in \mathbb N$, $q \geqslant 4$, $$\Omega = \left\lbrace (\alpha_1, \ldots, \alpha_q)\in \mathbb R^q \mathrel{\biggl|} \sum_{i=1}^q \alpha_i = 1,\ 
\alpha_1 \geqslant \alpha_2 \geqslant \ldots \geqslant \alpha_q\geqslant 0,\ \alpha_q
+\alpha_{q-1} \leqslant \alpha_1\right\rbrace.$$
Then $ d:=\max_{x\in \Omega} \Phi(x) = (q-3)+2\sqrt 2= q-0.1716\ldots$
\end{lemma}
\begin{proof}
We express $\alpha_1$ in terms of $\alpha_2, \ldots, \alpha_q$
and consider $$\Phi_0(\alpha_2, \ldots, \alpha_q):=\Phi\left(1-\sum_{i=2}^q \alpha_i, \alpha_2, 
\ldots, \alpha_q\right) = \frac{1}{\left(1-\sum_{i=2}^q \alpha_i\right)^{\left(1-\sum_{i=2}^q \alpha_i\right)} \alpha_2^{\alpha_2}
\ldots \alpha_q^{\alpha_q}}$$ on the segment
$$\Omega_0 = \left\lbrace
(\alpha_2, \ldots, \alpha_q) \mathrel{\Bigl|} \alpha_2\geqslant 0, \ldots,\ \alpha_q\geqslant 0,\ \alpha_2+\ldots+\alpha_{q-2}+ 2\alpha_{q-1} +2\alpha_q \leqslant 1 \right\rbrace.$$
Note that we have weakened the restrictions on $\alpha_i$.
However, we will see that $\max_{x\in \Omega} \Phi(x)=\max_{x\in \Omega_0} \Phi_0(x)$.

We have $\Phi_0(\alpha_2, \ldots, \alpha_q)=e^{-\left(1-\sum_{i=2}^q \alpha_i\right)\ln\left(1-\sum_{i=2}^q \alpha_i\right)-\sum_{i=2}^q(\alpha_i \ln \alpha_i) }$ and $$\frac{\partial \Phi_0}{\partial \alpha_k}(\alpha_2, \alpha_3,\ldots,\alpha_q) =
\left(\ln\left(1-\sum_{i=2}^q \alpha_i\right)-\ln\alpha_k \right)e^{-\left(1-\sum_{i=2}^q \alpha_i\right)\ln\left(1-\sum_{i=2}^q \alpha_i\right)-\sum_{i=2}^q(\alpha_i \ln \alpha_i) }.$$ Hence the only critical point is $(\alpha_2,\alpha_3,\ldots,\alpha_q)=\left(\frac{1}{q},\ldots, \frac{1}{q}\right)\notin \Omega_0$.
Therefore, $\Phi_0$ takes its maximal values on the border $\partial\Omega_0$ of $\Omega_0$.
Note that $\partial \Omega_0 = \Upsilon \cup \bigcup_{i=2}^q \Omega_i$
where $\Omega_i=\lbrace (\alpha_2,\ldots,\alpha_q)\in \Omega_0 \mid \alpha_i=0 \rbrace$
and $\Upsilon = \lbrace (\alpha_2,\ldots,\alpha_q)\in \Omega_0 \mid
 \alpha_2+\ldots+\alpha_{q-2}+ 2\alpha_{q-1} +2\alpha_q = 1 \rbrace$.
 Determining the critical points once again, we get $\Phi_0(x) \leqslant q -1$
 for all $x\in \bigcup_{i=2}^q \Omega_i$.

Consider $\Phi_0$ on $\Upsilon$. We express $\alpha_2$ in terms of $\alpha_3,\ldots,\alpha_q$
and define
\begin{equation*}\begin{split}\Phi_1(\alpha_3,\ldots,\alpha_q) = \Phi_0(1-\alpha_3-\ldots-\alpha_{q-2}-2\alpha_{q-1}-2\alpha_q, \alpha_3, \ldots, \alpha_q)=\\ \frac{1}{(\alpha_{q-1}+\alpha_q)^{\alpha_{q-1}+\alpha_q}
(1-\alpha_3-\ldots-\alpha_{q-2}-2\alpha_{q-1}-2\alpha_q)^{1-\alpha_3-\ldots-
\alpha_{q-2}-2\alpha_{q-1}-2\alpha_q}
\alpha_3^{\alpha_3}\ldots \alpha_{q}^{\alpha_q}}=\\
e^{-(\alpha_{q-1}+\alpha_q)\ln(\alpha_{q-1}+\alpha_q)
-(1-\alpha_3-\ldots-\alpha_{q-2}-2\alpha_{q-1}-2\alpha_q)
\ln(1-\alpha_3-\ldots-\alpha_{q-2}-2\alpha_{q-1}-2\alpha_q)
-\alpha_3\ln\alpha_3-\ldots-\alpha_q\ln \alpha_q}
\end{split}\end{equation*} on $$\Upsilon_1=\left\lbrace (\alpha_3,\ldots,\alpha_q)
\mid 1-\alpha_3-\ldots-\alpha_{q-2}-2\alpha_{q-1}-2\alpha_q \geqslant 0,\ \alpha_3 \geqslant 0, \ldots,\ \alpha_q \geqslant 0\right\rbrace.$$
Then
\begin{equation*}\begin{split}\frac{\partial \Phi_1}{\partial \alpha_i}(\alpha_3,\ldots,\alpha_q)=
(\ln(1-\alpha_3-\ldots-\alpha_{q-2}-2\alpha_{q-1}-2\alpha_q)-\ln\alpha_i)\cdot \\
e^{-(\alpha_{q-1}+\alpha_{q})\ln(\alpha_{q-1}+\alpha_{q})
-(1-\alpha_3-\ldots-
\alpha_{q-2}-2\alpha_{q-1}-2\alpha_q)
\ln(1-\alpha_3-\ldots-\alpha_{q-2}-2\alpha_{q-1}-2\alpha_q)-
\alpha_3\ln\alpha_3-\ldots-\alpha_q\ln \alpha_q}
\end{split}\end{equation*}
for $i=3,\ldots,q-2$ and
\begin{equation*}\begin{split}\frac{\partial \Phi_1}{\partial \alpha_i}(\alpha_3,\ldots,\alpha_q)=
(2\ln(1-\alpha_3-\ldots-\alpha_{q-2}-2\alpha_{q-1}-2\alpha_q)
-\ln(\alpha_{q-1}+\alpha_q)-\ln\alpha_i)\cdot \\
e^{-(\alpha_{q-1}+\alpha_q)\ln(\alpha_{q-1}+\alpha_q)
-(1-\alpha_3-\ldots-\alpha_{q-2}
-2\alpha_{q-1}-2\alpha_q)
\ln(1-\alpha_3-\ldots-\alpha_{q-2}-2\alpha_{q-1}-2\alpha_q)-
\alpha_3\ln\alpha_3-\ldots-\alpha_q\ln \alpha_q},
\end{split}\end{equation*}
for $i=q-1,\,q$.
Therefore, $(\tilde\alpha_3, \ldots, \tilde\alpha_q)$ where $\tilde\alpha_3=\ldots=\tilde\alpha_{q-2}=\frac{\sqrt 2}{4+(q-3)\sqrt 2}$
and $\tilde\alpha_{q-1}=\tilde\alpha_q=\frac{1}{4+(q-3)\sqrt 2}$, is
the critical point of $\Phi_1$. Hence $\Phi_1(\tilde\alpha_3, \ldots, \tilde\alpha_q)=(q-3)+2\sqrt 2= q-0.1716\ldots$
is the maximum of $\Phi_1$ on $\Upsilon_1$.

Now we return to the original variables.
Since $\tilde \alpha_1 = \tilde \alpha_{q-1}+\tilde \alpha_q = \frac{2}{4+(q-3)\sqrt 2}$
and $$\tilde \alpha_2 = 1-\alpha_3-\ldots-\alpha_{q-2}-2\alpha_{q-1}-2\alpha_q = \frac{\sqrt 2}{4+(q-3)\sqrt 2},$$
we have $\tilde\alpha_1 \geqslant \tilde\alpha_2 \geqslant \ldots \geqslant \tilde\alpha_q$.
Therefore $d=\Phi(\tilde\alpha_1, \ldots, \tilde\alpha_q)= (q-3)+2\sqrt 2= q-0.1716\ldots$ is the maximum
of $\Phi$ on $\Omega$.
\end{proof}

\begin{lemma}\label{LemmaTExampleIneqLambda}
Let $A = M_2(F)\oplus W$ (direct sum of left ideals) be an algebra 
over a field $F$ of characteristic $0$ and let $T=\lbrace t_1, t_2\rbrace$ be a semigroup of two elements.
Suppose there exists a linear isomorphism
$\varphi \colon W \mathrel{\widetilde{\rightarrow}} \langle \mathcal B_0\rangle_F \subset M_2(F)$
where ${\mathcal B}_0 \subseteq \lbrace e_{11}, e_{12}, e_{22} \rbrace$ is a subset such that $e_{12} \in \mathcal B_0$. Suppose the $T$-grading on $A$ is defined by $A^{(t_1)}=(M_2(F),0)$ and $A^{(t_2)}=\lbrace (\varphi(a), a) \mid a\in W\rbrace$.
Suppose $W M_2(F)=0$ and one of the following three conditions holds:
\begin{enumerate}
\item $M_2(F) W=0$ and $\varphi$ is a homomorphism of algebras;
\item $M_2(F) W=0$ and $W^2=0$;
\item $W^2=0$ and $\varphi$ is a homomorphism of left $M_2(F)$-modules.
\end{enumerate}
Then if $m(A, (FT)^*, \lambda)\ne 0$ for some $\lambda\vdash n$, $n\in\mathbb N$,
we have $\lambda_{q+1} = 0$ and $\lambda_{q-1}+\lambda_q \leqslant \lambda_1 + 1$
where $q :=\dim A$.
\end{lemma}
\begin{proof} In order to prove that $m(A, (FT)^*, \lambda)= 0$
for some $\lambda \vdash n$, it is sufficient to show that $e^*_{T_\lambda}f \equiv 0$
for every $f\in P^{(FT)^*}_n$ and a Young tableau~$T_\lambda$. Note that $e^*_{T_\lambda}f$ is alternating in the variables of each column of $T_\lambda$. Since $f$ is multilinear, it is sufficient to substitute only basis elements. However $\dim A = q$ and if $\lambda_{q+1} > 0$, then at least two of the basis elements corresponding to the variables of the first column coincide and $e^*_{T_\lambda}f$ vanish. Hence if $\lambda_{q+1} > 0$, we have $e_{T_\lambda}f \equiv 0$.

Consider in $A$ the homogeneous basis $$\mathcal B=\lbrace (e_{11},0), (e_{12},0), 
(e_{21},0), (e_{22},0)\rbrace \cup \lbrace (a, \varphi^{-1}(a) \mid a \in \mathcal B_0 \rbrace.$$
Note that the product of any two elements of $\mathcal B$ is either $0$ or again an element of $\mathcal B$.
Define the function $\theta \colon \mathcal B \to \mathbb Z$ by $\theta(e_{ij}, \varphi^{-1}(e_{ij}))=\theta(e_{ij},0)=j-i$. Let $a_1, \ldots, a_k \in\mathcal B$.
If $a_1 \ldots a_k \ne 0$, then \begin{equation}\label{EqTExampleThetaAk}-1 \leqslant \sum_{i=1}^k\theta(a_i)=\theta(a_1 \ldots a_k) \leqslant 1.\end{equation}
Note that $\sum_{b\in \mathcal B} \theta(b)=1$ and $\sum_{i=1}^{q-1} \theta(a_i) \geqslant 0$
for any different $a_i \in \mathcal B$. If $m(A, (FT)^*, \lambda)\ne 0$,
then $e^*_{T_\lambda}f \not\equiv 0$ for some $f\in P^{(FT)^*}_n$ and $e^*_{T_\lambda}f$
does not vanish under substitution of some basis elements $a_1, \ldots, a_n$. Again, $e^*_{T_\lambda}f$
is alternating in the variables of each column. In each of the first $\lambda_q$ columns we have $q$ boxes
and in each of the next $(\lambda_{q-1}-\lambda_q)$ columns we have $(q-1)$ boxes. Therefore, the impact
in $\sum_{i=1}^n \theta(a_i)$ of the basis elements corresponding to the first $\lambda_{q-1}$ columns is at least $\lambda_q$.  Since $\theta(a)=-1$ only for $a=(e_{21},0)$
and we cannot substitute more than one such element for the variables of the same column, by~(\ref{EqTExampleThetaAk}) there must be at least $(\lambda_q-1)$ other columns. Hence $\lambda_1-\lambda_{q-1} \geqslant \lambda_q-1$
and we get the lemma.
\end{proof}

\begin{lemma}\label{LemmaTExampleLower}
Let $A$ be an algebra from Lemma~\ref{LemmaTExampleIneqLambda}.
Suppose that for every $\lambda \vdash n$, $n\in\mathbb N$, such that
$\lambda_{q-1}+\lambda_q \leqslant \lambda_1$, we have
$m(A, (FT)^*, \lambda)\ne 0$. Then there exists  $$\lim\limits_{n\to \infty} \sqrt[n]{c_n^{T\text{-}\mathrm{gr}}(A)} =(q-3)+2\sqrt 2= q - 0.1716\ldots$$
\end{lemma}
\begin{proof}
Let $$\Omega = \left\lbrace (\alpha_1, \ldots, \alpha_q)\in \mathbb R^q \mathrel{\biggl|} \sum_{i=1}^q \alpha_i = 1,\ 
\alpha_1 \geqslant \alpha_2 \geqslant \ldots \geqslant \alpha_q\geqslant 0,\ \alpha_q
+\alpha_{q-1} \leqslant \alpha_1\right\rbrace.$$
By Lemma~\ref{LemmaMaxTExample},
$ d:=\max_{x\in \Omega} \Phi(x) = (q-3)+2\sqrt 2= q-0.1716\ldots$
Denote by $(\alpha_1, \ldots, \alpha_q) \in \Omega$ such a point that
$\Phi(\alpha_1, \ldots, \alpha_q)=d$.
For every $n\in\mathbb N$ define $\mu\vdash n$ by
$\mu_i = [\alpha_i n]$ for $2\leqslant i \leqslant q$
and $\mu_1 =  n-\sum_{i=1}^q \mu_i$.
For every $\varepsilon > 0$ there exists $n_0\in\mathbb N$
such that for every $n\geqslant n_0$
we have $\Phi\left(\frac{\mu_1}{n},\ldots,\frac{\mu_q}{n}\right) > d-\varepsilon$.
 By the assumptions of the lemma, $m(A,(FT)^*,\mu) \ne 0$
 and by the hook and the Stirling formulas, there exist $C_1 > 0$ and $r_1\in\mathbb R$ such that
 we have \begin{equation}\begin{split} c^{(FT)^*}_n(A) \geqslant \dim M(\mu) = \frac{n!}{\prod_{i,j} h_{ij}}
  \geqslant \frac{n!}{(\mu_1+q-1)! \ldots (\mu_q+q-1)!} \geqslant \\
  \frac{n!}{n^{q(q-1)}\mu_1! \ldots \mu_q!} \geqslant
  \frac{C_1 n^{r_1} 
\left(\frac{n}{e}\right)^n}{\left(\frac{\mu_1}{e}\right)^{\mu_1}\ldots
\left(\frac{\mu_q}{e}\right)^{\mu_q}}\geqslant \\ C_1 n^{r_1}\left(\frac{1}
{\left(\frac{\mu_1}{n}\right)^{\frac{\mu_1}{n}}\ldots
\left(\frac{\mu_q}{n}\right)^{\frac{\mu_q}{n}}}\right)^n \geqslant C_1 n^{r_1}
 (d-\varepsilon)^n.\end{split}\end{equation}
 Hence $\mathop{\underline\lim}_{n\to\infty}\sqrt[n]{c_n^{(FT)^*}(A)}
\geqslant d-\varepsilon$. Since $\varepsilon > 0$ is arbitrary, $\mathop{\underline\lim}_{n\to\infty}\sqrt[n]{c_n^{(FT)^*}(A)}
\geqslant d$. Now Lemmas~\ref{LemmaCnGrCnGenH}, \ref{LemmaTExampleUpperFd}, \ref{LemmaMaxTExample}, and \ref{LemmaTExampleIneqLambda}
finish the proof.
\end{proof}

\section{A $T_1$-graded algebra with a non-integer graded PI-exponent}

\begin{theorem}\label{TheoremT1GradFractPI}
Let $A = M_2(F)\oplus \UT_2(F)$ (direct sum of ideals) where $F$ is a field of characteristic $0$. Define a $T_1$-grading on $A$ by $A^{(0)}=(M_2(F),0)$, $A^{(1)}=\lbrace (\varphi(a), a) \mid a \in \UT_2(F) \rbrace$ where $\varphi \colon \UT_2(F) \hookrightarrow M_2(F)$ is the natural embedding. In other words, $A$ is an algebra from Example~\ref{ExampleT1} for $k=2$. Then there exists  $\lim\limits_{n\to \infty} \sqrt[n]{c_n^{T_1\text{-}\mathrm{gr}}(A)} =4+2\sqrt 2= 6.8284\ldots$
\end{theorem}

To prove Theorem~\ref{TheoremT1GradFractPI}, we need the following lemma.  We omit $\varphi$
for shortness and write $(e_{ij}, e_{ij})$ instead of $(e_{ij}, \varphi^{-1}(e_{ij}))$.
\begin{lemma}\label{LemmaAltT1}
Let $\lambda \vdash n$, $n\in\mathbb N$, $\lambda_8 = 0$, and $\lambda_6+\lambda_7 \leqslant \lambda_1$.
Then $m(A,(FT_1)^*,\lambda) \ne 0$.
\end{lemma}
\begin{proof}
It is sufficient to show that for some $f\in P_n$ and some $T_\lambda$
we have $e_{T_\lambda}f \not\equiv 0$ on $A$.

Note that each $e_{T_\lambda}f$ is alternating in $\lambda_7$ disjoint sets of variables, each of $7$ variables. By the proof Lemma~\ref{LemmaTExampleIneqLambda}, each column will have the impact at least $1$ to the sum of the values of $\theta$ on the elements substituted for the variables of $e_{T_\lambda}f$.
Therefore, we need a compensation.

Let $\beta_2=\lambda_6-\lambda_7$. Fix numbers $\beta_3,\ldots,\beta_{12} \geqslant 0$
such that $\beta_3+\beta_5+\beta_7+\beta_9+\beta_{11} = \lambda_7$,\quad
$\beta_3+\beta_4=\lambda_5-\lambda_6$,\quad $\beta_5+\beta_6=\lambda_4-\lambda_5$,
$\beta_7+\beta_8=\lambda_3-\lambda_4$,\quad $\beta_9+\beta_{10}=\lambda_2-\lambda_3$,
and $\beta_{11}+\beta_{12}=\lambda_1-\lambda_2$. In other words, we have
$$D_\lambda=\begin{array}{|c|c|c|c|c|c|c|c|c|c|c|c|}
\multicolumn{1}{c}{\lambda_7} & \multicolumn{1}{c}{\beta_2} & \multicolumn{1}{c}{\beta_3} & \multicolumn{1}{c}{\beta_4} & \multicolumn{1}{c}{\beta_5} & \multicolumn{1}{c}{\beta_6} & \multicolumn{1}{c}{\beta_7} & \multicolumn{1}{c}{\beta_8} & \multicolumn{1}{c}{\beta_9} &
\multicolumn{1}{c}{\beta_{10}} & \multicolumn{1}{c}{\beta_{11}} & \multicolumn{1}{c}{\beta_{12}} \\
\hline
 \ldots & \ldots & \ldots & \ldots & \ldots & \ldots & \ldots & \ldots & \ldots & \ldots & \ldots & \ldots \\
 \cline{1-12}
 \ldots & \ldots & \ldots & \ldots & \ldots & \ldots & \ldots & \ldots & \ldots & \ldots \\
 \cline{1-10}
 \ldots & \ldots & \ldots & \ldots & \ldots & \ldots & \ldots & \ldots \\
 \cline{1-8}
 \ldots & \ldots & \ldots & \ldots & \ldots & \ldots \\
 \cline{1-6}
 \ldots & \ldots & \ldots & \ldots \\
 \cline{1-4}
 \ldots & \ldots \\
 \cline{1-2}
 \ldots  \\
 \cline{1-1}
\end{array}.$$
(Here $\beta_i$ denotes the number of columns in each block of columns.)

Each of the first $\lambda_7$ columns will give an impact $1$ to $\theta$,
which will be compensated by the columns that are marked by $\beta_3$, $\beta_5$, $\beta_7$, $\beta_9$, and $\beta_{11}$. The columns that are marked by $\beta_2$, $\beta_3$, $\beta_4$, $\beta_6$, $\beta_8$, $\beta_{10}$, and $\beta_{12}$ will give zero impact to $\theta$.

We fix some Young tableau $T_\lambda$ of the shape $\lambda$ filled in with the numbers
from $1$ to $n$.
For each column of $T_\lambda$ we define a multilinear alternating polynomial depending on the variables with the indexes from the column. For shortness, we denote the polynomials
corresponding to different columns in the $i$th block by the same letter $f_i$.
By $(i_1, \ldots, i_\ell)$ we denote the $\ell$-tuple of numbers from a column (from up to down).
By $S\lbrace i_1, \ldots, i_\ell\rbrace$ we denote the symmetric group on $i_1, \ldots, i_\ell$.
We define $$f_1 := \sum_{\sigma\in S\lbrace i_1, \ldots, i_7\rbrace} (\sign \sigma)
x^{h_0}_{\sigma(i_3)}
x^{h_1}_{\sigma(i_2)}
x^{h_0}_{\sigma(i_6)}
x^{h_1}_{\sigma(i_4)}
x^{h_0}_{\sigma(i_5)}
x^{h_0}_{\sigma(i_1)}
x^{h_1}_{\sigma(i_7)},
$$ 
$$f_2 := \sum_{\sigma\in S\lbrace i_1, \ldots, i_6\rbrace} (\sign \sigma)
x^{h_0}_{\sigma(i_3)}
x^{h_1}_{\sigma(i_2)}
x^{h_0}_{\sigma(i_6)}
x^{h_1}_{\sigma(i_4)}
x^{h_0}_{\sigma(i_5)}
x^{h_0}_{\sigma(i_1)}
,$$ 
$$f_3 := \sum_{\sigma\in S\lbrace i_1, \ldots, i_5\rbrace} (\sign \sigma)
x^{h_0}_{\sigma(i_5)}
x^{h_1}_{\sigma(i_4)}
x^{h_0}_{\sigma(i_1)}
x^{h_1}_{\sigma(i_2)}
x^{h_0}_{\sigma(i_3)}
,$$
$$f_4 := \sum_{\sigma\in S\lbrace i_1, \ldots, i_5\rbrace} (\sign \sigma)
x^{h_1}_{\sigma(i_2)}
x^{h_0}_{\sigma(i_3)}
x^{h_1}_{\sigma(i_5)}
x^{h_1}_{\sigma(i_4)}
x^{h_0}_{\sigma(i_1)}
,$$
$$f_5 := \sum_{\sigma\in S\lbrace i_1, \ldots, i_4\rbrace} (\sign \sigma)
x^{h_1}_{\sigma(i_4)}
x^{h_0}_{\sigma(i_1)}
x^{h_1}_{\sigma(i_2)}
x^{h_0}_{\sigma(i_3)}
,$$ $$f_6 := \sum_{\sigma\in S\lbrace i_1, \ldots, i_4\rbrace} (\sign \sigma)
x^{h_1}_{\sigma(i_2)}
x^{h_1}_{\sigma(i_3)}
x^{h_1}_{\sigma(i_4)}
x^{h_0}_{\sigma(i_1)}
,$$ $$f_7 := \sum_{\sigma\in S\lbrace i_1, i_2, i_3\rbrace} (\sign \sigma)
x^{h_0}_{\sigma(i_1)}
x^{h_1}_{\sigma(i_2)}
x^{h_0}_{\sigma(i_3)}
,\qquad f_8 := \sum_{\sigma\in S\lbrace i_1, i_2, i_3\rbrace} (\sign \sigma)
x^{h_1}_{\sigma(i_2)}
x^{h_1}_{\sigma(i_3)}
x^{h_0}_{\sigma(i_1)}
,$$ $$f_9 := \sum_{\sigma\in S\lbrace i_1, i_2\rbrace} (\sign \sigma)
x^{h_0}_{\sigma(i_1)}
x^{h_1}_{\sigma(i_2)}
,\qquad f_{10} := \sum_{\sigma\in S\lbrace i_1, i_2\rbrace} (\sign \sigma)
x^{h_0}_{\sigma(i_2)}
x^{h_0}_{\sigma(i_1)}
,$$ $$f_{11} := x^{h_0}_{i_1},\qquad f_{12} := x^{h_1}_{i_1}.$$

Define the polynomial
$$f=(f_1 f_3)^{\beta_3}(f_1 f_5)^{\beta_5}(f_1 f_7)^{\beta_7}(f_1 f_9)^{\beta_9}(f_1 f_{11})^{\beta_{11}}
f_2^{\beta_2} f_4^{\beta_4}f_6^{\beta_6}f_8^{\beta_8}f_{10}^{\beta_{10}}f_{12}^{\beta_{12}} \in P_n.$$
 As we have already mentioned, here different copies of $f_i$ depend on different variables.

The copies of $f_1$ are alternating polynomials
of degree $7$ corresponding to the first $\lambda_7$ columns of height $7$.

The copies of $f_2$ are alternating polynomials of degree $6$ corresponding to
the next $\beta_2$ columns of height $6$.

\ldots

The copies of $f_{12}$ are polynomials of degree $1$ with the indexes from the last
$\beta_{12}$ columns of height $1$.

We claim that $e_{T_\lambda}f \not\equiv 0$. In order to verify this, we fill $D_\lambda$ with specific homogeneous elements and denote the tableau obtained
by $\tau$. (See Figure~\ref{FigureTauT1}.)
\begin{landscape}

\begin{figure}\caption{Substitution for the variables of $e_{T_\lambda}f$, grading semigroup $T_1$}\label{FigureTauT1}
$$\tau=\begin{array}{|c|c|c|c|c|c|c|c|c|c|c|c|}
\multicolumn{1}{c}{\lambda_7} & \multicolumn{1}{c}{\beta_2} & \multicolumn{1}{c}{\beta_3} & \multicolumn{1}{c}{\beta_4} & \multicolumn{1}{c}{\beta_5} & \multicolumn{1}{c}{\beta_6} & \multicolumn{1}{c}{\beta_7} & \multicolumn{1}{c}{\beta_8} & \multicolumn{1}{c}{\beta_9} &
\multicolumn{1}{c}{\beta_{10}} & \multicolumn{1}{c}{\beta_{11}} & \multicolumn{1}{c}{\beta_{12}} \\
\hline
 (e_{21},0) & (e_{21},0) & (e_{21},0) & (e_{21},0) & (e_{21},0) & (e_{21},0) & (e_{21},0) & (e_{21},0) & (e_{21},0) & (e_{21},0) & (e_{21},0) & (e_{11}, e_{11}) \\
 \cline{1-12}
 (e_{11}, e_{11}) & (e_{11}, e_{11}) & (e_{11}, e_{11}) & (e_{11}, e_{11}) & (e_{11}, e_{11}) & (e_{11}, e_{11}) & (e_{11}, e_{11}) & (e_{11}, e_{11}) & (e_{11}, e_{11}) & (e_{12}, 0) \\
 \cline{1-10}
 (e_{11}, 0) & (e_{11}, 0) & (e_{11}, 0) & (e_{11}, 0) & (e_{11}, 0) & (e_{12}, e_{12}) & (e_{11}, 0) & (e_{12}, e_{12}) \\
 \cline{1-8}
 (e_{22}, e_{22}) & (e_{22}, e_{22}) & (e_{22}, e_{22}) & (e_{22}, e_{22}) & (e_{22}, e_{22}) & (e_{22}, e_{22}) \\
 \cline{1-6}
 (e_{22}, 0) & (e_{22}, 0) & (e_{22}, 0) & (e_{12}, e_{12}) \\
 \cline{1-4}
 (e_{12}, 0) & (e_{12}, 0) \\
 \cline{1-2}
 (e_{12}, e_{12})  \\
 \cline{1-1}
\end{array}$$ 
\end{figure}

 (Here in the $i$th block we have $\beta_i$ columns with the same values
in all cells of a row. For shortness, we depict each value for each block only once.
The tableau $\tau$ is still of the shape $\lambda$.)

\end{landscape}

Now for each variable we substitute the element from the corresponding box in $\tau$.
Note that $f$ does not vanish under this substitution.

Recall that $e_{T_\lambda} = a_{T_\lambda}b_{T_\lambda}$
where $a_{T_\lambda}$ is the symmetrization in the variables of each row
and $b_{T_\lambda}$ is the alternation in the variables of each column.
Since all $f_i$ are alternating polynomials, $b_{T_\lambda} f$
is a nonzero multiple of $f$.

Two sets of variables correspond to the second row of $T_\lambda$.
For the variables of the first group we substitute $(e_{11},e_{11})  \in A^{(1)}$,
for the second one, we substitute $(e_{12},0) \in A^{(0)}$.
Thus if an item in $a_{T_\lambda}$ mixes variables from these two groups,
at least one variable from the second group, i.e. in $f_{11}$, is replaced with a variable form the first one. However, $f_{10}$ vanishes if at least one variable of it is replaced with an element of $A^{(1)}$
since $h_0$ is applied for both variables of $f_{10}$.
Thus all items in $a_{T_\lambda}b_{T_\lambda}f$ where variables from these two groups are mixed, vanish.

Therefore, if an item in $a_{T_\lambda}$ replaces a variable from the first two columns with
a variable with a different value from the tableau $\tau$, we will have too many elements from
 $A^{(1)}$ substituted for the variables of $f_1$ and $f_2$ and the result is zero in virtue of the action of $h_0$.
 Therefore all items in $a_{T_\lambda}b_{T_\lambda}f$ where variables from the first two columns having different values are mixed, vanish. We continue this procedure and finally show that if an item in $a_{T_\lambda}$ does not
 stabilize the sets of variables with the same values from the tableau $\tau$, the corresponding
 item in $a_{T_\lambda}b_{T_\lambda}f$ vanishes.
 Hence the value of $a_{T_\lambda}b_{T_\lambda}f$ is a nonzero multiple of the value of $b_{T_\lambda}f$, i.e. is nonzero. The lemma is proved.
\end{proof}

\begin{proof}[Proof of Theorem~\ref{TheoremT1GradFractPI}.]
We use Lemmas~\ref{LemmaTExampleLower} and~\ref{LemmaAltT1}.
\end{proof}

\section{A $T_2$-graded algebra with a non-integer graded PI-exponent}

\begin{theorem}\label{TheoremT2GradFractPI}
Let $A_2 = M_2(F)\oplus F j_{11} \oplus F j_{12} \oplus F j_{22}$ (direct sum of ideals) where $F$ is a field of characteristic $0$ and
$j_{11}^2=j_{12}^2=j_{22}^2=0$. Define a $T_2$-grading on $A_2$ by $A_2^{(0)}=(M_2(F),0)$, $A_2^{(v)}=\langle (e_{11}, j_{11}), (e_{12}, j_{12}), (e_{22}, j_{22})\rangle$. Then there exists  $\lim\limits_{n\to \infty} \sqrt[n]{c_n^{T_2\text{-}\mathrm{gr}}(A_2)} =
4+2\sqrt 2= 6.8284\ldots$.
\end{theorem}
\begin{proof}
Let $A$ be the algebra from Theorem~\ref{TheoremT1GradFractPI}. Define a linear isomorphism
$\psi \colon A \mathrel{\widetilde{\rightarrow}} A_2$ by $\psi(e_{ij}, e_{k\ell})=(e_{ij}, j_{k\ell})$
and $\psi(e_{ij}, 0)=(e_{ij}, 0)$.
Then $\psi(A^{(0)})=A_2^{(0)}$ and $\psi(A^{(1)})=A_2^{(v)}$.
 Define an isomorphism $\Theta \colon F\langle X^{T_1\text{-}\mathrm{gr}} \rangle
\to F\langle X^{T_2\text{-}\mathrm{gr}} \rangle$ of algebras
by $\Theta(x^{(0)}_i)=x^{(0)}_i$ and $\Theta(x^{(1)}_i)=x^{(v)}_i$.
We claim that \begin{equation}\label{EqIdT1IdT2}\Theta(\Id^{T_1\text{-}\mathrm{gr}}(A))=\Id^{T_2\text{-}\mathrm{gr}}(A_2).\end{equation}
Since $F$ is of characteristic $0$, both $\Id^{T_1\text{-}\mathrm{gr}}(A)$
and $\Id^{T_2\text{-}\mathrm{gr}}(A_2)$ are generated by multilinear polynomials.
(The proof is completely analogous to~\cite[Theorem~1.3.8]{ZaiGia}.)
In other words, in order to prove~(\ref{EqIdT1IdT2}), it is sufficient to show that if $f\in F\langle X^{T_1\text{-}\mathrm{gr}} \rangle$ is multilinear as an ordinary polynomial in variables
$x_1^{(t_1)}, x_2^{(t_2)},\ldots, x_n^{(t_n)}$ where $t_i\in T_1$, then
$f \in \Id^{T_1\text{-}\mathrm{gr}}(A)$ if and only if $\Theta(f)\in
\Id^{T_2\text{-}\mathrm{gr}}(A_2)$.
We substitute only homogeneous elements. Note that
$\pi \psi(a) = \pi(a)$ where $\pi$ is the projection on the first component: $\pi(a,b)=a$ for all $(a,b)\in A$ and $(a,b)\in A_2$. Moreover, if $a \in A^{(0)} \cup A^{(1)}$, then $\pi(a)=0$ if and only if $a=0$. Since $T_1$ and $T_2$ are commutative, the value of $f$ 
under the substitution of homogeneous
elements is again a homogeneous element. Applying $\pi$, we show that $f\in P_n^{T_1\text{-}\mathrm{gr}}$
vanishes under the substitution of homogeneous
elements $a^{(t_i)}_i\in A^{(t_i)}$, $t_i\in T_1$, if and only if $\Theta(f)$ vanishes under the substitution of $\psi\left(a^{(\tilde t_i)}_i\right)\in A_2^{(\tilde t_i)}$. (Here $\tilde 0 = 0$
and $\tilde 1 = v$.) Hence~(\ref{EqIdT1IdT2}) holds and
$$c_n^{T_1\text{-}\mathrm{gr}}(A)=c_n^{T_2\text{-}\mathrm{gr}}(A_2)\text{ for all }n\in\mathbb N.$$
Now we apply Theorem~\ref{TheoremT1GradFractPI}.
\end{proof}

\section{A $T_3$-graded algebra with a non-integer graded PI-exponent}

\begin{theorem}\label{TheoremT3GradFractPI}
Let $F$ be a field of characteristic $0$.
Denote by $I$ the irreducible left $M_2(F)$-module isomorphic to the minimal left ideal $\langle e_{12}, e_{22}\rangle_F \subset M_2(F)$.
Let $A = M_2(F)\oplus I$ (direct sum of left ideals) where $IM_2(F) := 0$ and $I^2:=0$.
 Define a $T_3$-grading on $A$ by $A^{(e_1)}=(M_2(F),0)$, $A^{(e_2)}=\lbrace (\varphi(a), a) \mid a \in I \rbrace$ where $\varphi \colon I \hookrightarrow M_2(F)$ is the natural embedding which is a homomorphism of $ M_2(F)$-modules. Then there exists  $\lim\limits_{n\to \infty} \sqrt[n]{c_n^{T_3\text{-}\mathrm{gr}}(A)} =3+2\sqrt 2= 5.8284\ldots$
\end{theorem}

Theorem~\ref{TheoremT3GradFractPI} is proved at the end of the section.

\begin{remark}
Similarly, $\lim\limits_{n\to \infty} \sqrt[n]{c_n^{T_3^{\,\mathrm{op}}\text{-}\mathrm{gr}}(A^{\,\mathrm{op}})} =3+2\sqrt 2= 5.8284\ldots$
\end{remark}
\begin{remark}
Note that $A$ does not contain unity. If a $T_3$-graded algebra is unital, its graded PI-exponent
is integer. (See Theorem~\ref{TheoremTIdemAmitsur} below.)
\end{remark}
\begin{remark}
The algebra $A$ is $T_3$-graded-simple (see Proposition~\ref{PropositionAT3GrSimple} below), however $3+2\sqrt 2=\PIexp^{T_3\text{-}\mathrm{gr}}(A)<\dim A=6$
even if $F$ is algebraically closed.
\end{remark}

\begin{proposition}\label{PropositionAT3GrSimple}
The algebra $A$ is $T_3$-graded-simple.
\end{proposition}
\begin{proof} First, we notice that $J(A)=I$.
Suppose $W\ne 0$ is a graded ideal of $A$.
Then there exists a nonzero homogeneous element $(a_1, b_1) \in W$ where $a_1 \in M_2(F)$ and $b_1\in I$.  Since $I$ does not contain homogeneous elements, we have $a_1 \ne 0$.
However $(a_1,b_1)(E,0) = (a_1,0) \in (M_2(F),0) \cap W$.
Since $W$ is a two sided ideal and $M_2(F)$ is simple, we get $(M_2(F),0) \subseteq W$.
Furthermore $(0,I) = (M_k(F),0) (0,I) \subseteq W$. Thus $W = A$, and $A$ is $T_3$-graded-simple.
\end{proof}

To prove Theorem~\ref{TheoremT3GradFractPI}, we need Lemma~\ref{LemmaAltT3} which is an analog of Lemma~\ref{LemmaAltT1}.
We omit $\varphi$
for shortness and write $(e_{ij}, e_{ij})$ instead of $(e_{ij}, \varphi^{-1}(e_{ij}))$. 

\begin{lemma}\label{LemmaAltT3}
Let $\lambda \vdash n$, $n\in\mathbb N$, $\lambda_7 = 0$, and $\lambda_5+\lambda_6 \leqslant \lambda_1$.
Then $m(A,(FT_3)^*,\lambda) \ne 0$.
\end{lemma}
\begin{proof}
It is sufficient to show that for some $f\in P_n$ and some $T_\lambda$
we have $e_{T_\lambda}f \not\equiv 0$ on $A$.

Let $\beta_2=\lambda_5-\lambda_6$. Fix numbers $\beta_3,\ldots,\beta_{10} \geqslant 0$
such that $\beta_3+\beta_5+\beta_7+\beta_9 = \lambda_6$,\quad
$\beta_3+\beta_4=\lambda_4-\lambda_5$,\quad $\beta_5+\beta_6=\lambda_3-\lambda_4$,\quad
$\beta_7+\beta_8=\lambda_2-\lambda_3$, and $\beta_9+\beta_{10}=\lambda_1-\lambda_2$.
 In other words, we have
$$D_\lambda=\begin{array}{|c|c|c|c|c|c|c|c|c|c|}
\multicolumn{1}{c}{\lambda_6} & \multicolumn{1}{c}{\beta_2} & \multicolumn{1}{c}{\beta_3} & \multicolumn{1}{c}{\beta_4} & \multicolumn{1}{c}{\beta_5} & \multicolumn{1}{c}{\beta_6} & \multicolumn{1}{c}{\beta_7} & \multicolumn{1}{c}{\beta_8} & \multicolumn{1}{c}{\beta_9} &
\multicolumn{1}{c}{\beta_{10}} \\
\hline
 \ldots & \ldots & \ldots & \ldots & \ldots & \ldots & \ldots & \ldots & \ldots & \ldots \\
 \cline{1-10}
 \ldots & \ldots & \ldots & \ldots & \ldots & \ldots & \ldots & \ldots \\
 \cline{1-8}
 \ldots & \ldots & \ldots & \ldots & \ldots & \ldots \\
 \cline{1-6}
 \ldots & \ldots & \ldots & \ldots \\
 \cline{1-4}
 \ldots & \ldots \\
 \cline{1-2}
 \ldots  \\
 \cline{1-1}
\end{array}.$$
(Here $\beta_i$ denotes the number of columns in each block of columns.)

We fix some Young tableau $T_\lambda$ of the shape $\lambda$ filled in with the numbers
from $1$ to $n$.
Like in Lemma~\ref{LemmaAltT1}, for each column of $T_\lambda$ we define a multilinear alternating polynomial depending on the variables with the indexes from the column. For shortness, we denote the polynomials
corresponding to different columns in the $i$th block by the same letter $f_i$.
By $(i_1, \ldots, i_\ell)$ we denote the $\ell$-tuple of numbers from a column (from up to down).
By $S\lbrace i_1, \ldots, i_\ell\rbrace$ we denote the symmetric group on $i_1, \ldots, i_\ell$.
We define $$f_1 := \sum_{\sigma\in S\lbrace i_1, \ldots, i_6\rbrace} (\sign \sigma)
x^{h_{e_1}}_{\sigma(i_3)}
x^{h_{e_2}}_{\sigma(i_5)}
x^{h_{e_1}}_{\sigma(i_4)}
x^{h_{e_2}}_{\sigma(i_2)}
x^{h_{e_1}}_{\sigma(i_1)}
x^{h_{e_1}}_{\sigma(i_6)}
,$$ 
$$f_2 := \sum_{\sigma\in S\lbrace i_1, \ldots, i_5\rbrace} (\sign \sigma)
x^{h_{e_1}}_{\sigma(i_1)}
x^{h_{e_1}}_{\sigma(i_3)}
x^{h_{e_2}}_{\sigma(i_5)}
x^{h_{e_2}}_{\sigma(i_2)}
x^{h_{e_1}}_{\sigma(i_4)}
,$$
$$f_3 := \sum_{\sigma\in S\lbrace i_1, \ldots, i_4\rbrace} (\sign \sigma)
x^{h_{e_1}}_{\sigma(i_4)}
x^{h_{e_2}}_{\sigma(i_2)}
x^{h_{e_1}}_{\sigma(i_1)}
x^{h_{e_1}}_{\sigma(i_3)}
,$$ $$f_4 := \sum_{\sigma\in S\lbrace i_1, \ldots, i_4\rbrace} (\sign \sigma)
x^{h_{e_1}}_{\sigma(i_1)}
x^{h_{e_1}}_{\sigma(i_3)}
x^{h_{e_2}}_{\sigma(i_4)}
x^{h_{e_2}}_{\sigma(i_2)}
,$$ $$f_5 := \sum_{\sigma\in S\lbrace i_1, i_2, i_3\rbrace} (\sign \sigma)
x^{h_{e_2}}_{\sigma(i_2)}
x^{h_{e_1}}_{\sigma(i_1)}
x^{h_{e_1}}_{\sigma(i_3)}
,\qquad f_6 := \sum_{\sigma\in S\lbrace i_1, i_2, i_3\rbrace} (\sign \sigma)
x^{h_{e_2}}_{\sigma(i_2)}
x^{h_{e_1}}_{\sigma(i_1)}
x^{h_{e_2}}_{\sigma(i_3)}
,$$ $$f_7 := \sum_{\sigma\in S\lbrace i_1, i_2\rbrace} (\sign \sigma)
x^{h_{e_2}}_{\sigma(i_2)}
x^{h_{e_1}}_{\sigma(i_1)}
,\qquad f_8 := \sum_{\sigma\in S\lbrace i_1, i_2\rbrace} (\sign \sigma)
x^{h_{e_1}}_{\sigma(i_1)}
x^{h_{e_1}}_{\sigma(i_2)}
,$$ $$f_{9} := x^{h_{e_1}}_{i_1},\qquad f_{10} := x^{h_{e_2}}_{i_1}.$$

Define the polynomial
$$f=(f_3 f_1)^{\beta_3}(f_5 f_1)^{\beta_5}(f_7 f_1)^{\beta_7}(f_9 f_1)^{\beta_9}
f_2^{\beta_2} f_4^{\beta_4}f_6^{\beta_6}f_8^{\beta_8}f_{10}^{\beta_{10}} \in P_n.$$
 As we have already mentioned, here different copies of $f_i$ depend on different variables.

The copies of $f_1$ are alternating polynomials
of degree $6$ corresponding to the first $\lambda_6$ columns of height $6$.

The copies of $f_2$ are alternating polynomials of degree $5$ corresponding to
the next $\beta_2$ columns of height $5$.

\ldots

The copies of $f_{10}$ are polynomials of degree $1$ with the indexes from the last
$\beta_{10}$ columns of height $1$.

We claim that $e_{T_\lambda}f \not\equiv 0$. In order to verify this, we fill $D_\lambda$ with specific homogeneous elements and denote the tableau obtained
by $\tau$. (See Figure~\ref{FigureTauT3}.)
\begin{landscape}

\begin{figure}\caption{Substitution for the variables of $e_{T_\lambda}f$, grading semigroup $T_3$}\label{FigureTauT3}
$$\tau=\begin{array}{|c|c|c|c|c|c|c|c|c|c|}
\multicolumn{1}{c}{\lambda_7} & \multicolumn{1}{c}{\beta_2} & \multicolumn{1}{c}{\beta_3} & \multicolumn{1}{c}{\beta_4} & \multicolumn{1}{c}{\beta_5} & \multicolumn{1}{c}{\beta_6} & \multicolumn{1}{c}{\beta_7} & \multicolumn{1}{c}{\beta_8} & \multicolumn{1}{c}{\beta_9} &
\multicolumn{1}{c}{\beta_{10}} \\
\hline
 (e_{21},0) & (e_{21},0) & (e_{21},0) & (e_{21},0) & (e_{21},0) & (e_{21},0) & (e_{21},0) & (e_{21},0) & (e_{21},0) & (e_{22}, e_{22}) \\
 \cline{1-10}
(e_{22}, e_{22}) & (e_{22}, e_{22}) & (e_{22}, e_{22}) & (e_{22}, e_{22}) & (e_{22}, e_{22}) & (e_{22}, e_{22}) & (e_{22}, e_{22}) & (e_{12}, 0) \\
 \cline{1-8}
 (e_{11}, 0) & (e_{11}, 0) & (e_{11}, 0) & (e_{11}, 0) & (e_{11}, 0) & (e_{12}, e_{12})  \\
 \cline{1-6}
 (e_{22}, 0) & (e_{22}, 0) & (e_{22}, 0) & (e_{12}, e_{12}) \\
 \cline{1-4}
 (e_{12}, e_{12}) & (e_{12}, e_{12}) \\
 \cline{1-2}
 (e_{12}, 0)  \\
 \cline{1-1}
\end{array}$$ 
\end{figure}

 (Here in the $i$th block we have $\beta_i$ columns with the same values
in all cells of a row. For shortness, we depict each value for each block only once.
The tableau $\tau$ is still of the shape $\lambda$.)

\end{landscape}

Now for each variable we substitute the element from the corresponding box in $\tau$.
Note that $f$ does not vanish under this substitution.

Recall that $e_{T_\lambda} = a_{T_\lambda}b_{T_\lambda}$
where $a_{T_\lambda}$ is the symmetrization in the variables of each row
and $b_{T_\lambda}$ is the alternation in the variables of each column.
Since all $f_i$ are alternating polynomials, $b_{T_\lambda} f$
is a nonzero multiple of $f$.

Two sets of variables correspond to the second row of $T_\lambda$.
For the variables of the first group we substitute $(e_{22},e_{22})  \in A^{(e_2)}$,
for the second one, we substitute $(e_{12},0) \in A^{(e_1)}$.
Thus if an item in $a_{T_\lambda}$ mixes variables from these two groups,
at least one variable from the second group, i.e. in $f_8$, is replaced with a variable form the first one. However, $f_8$ vanishes if at least one variable of it is replaced with an element of $A^{(e_2)}$
since $h_{e_1}$ is applied for both variables of $f_8$.
Thus all items in $a_{T_\lambda}b_{T_\lambda}f$ where variables from these two groups are mixed, vanish.

Therefore, if an item in $a_{T_\lambda}$ replaces a variable from the first three columns with
a variable with a different value from the tableau $\tau$, we have too many elements from
 $A^{(e_2)}$ substituted for the variables of $f_1$, $f_2$, and $f_3$ and the result is zero in virtue of the action of $h_{e_1}$.
 Therefore all items in $a_{T_\lambda}b_{T_\lambda}f$ where variables from the first three columns having different values are mixed, vanish. We continue this procedure and finally show that if an item in $a_{T_\lambda}$ does not
 stabilize the sets of variables with the same values from the tableau $\tau$, the corresponding
 item in $a_{T_\lambda}b_{T_\lambda}f$ vanishes.
 Hence the value of $a_{T_\lambda}b_{T_\lambda}f$ is a nonzero multiple of the value of $b_{T_\lambda}f$, i.e. is nonzero. The lemma is proved.
\end{proof}
\begin{proof}[Proof of Theorem~\ref{TheoremT3GradFractPI}.]
We use Lemmas~\ref{LemmaTExampleLower} and~\ref{LemmaAltT3}.
\end{proof}

\section{Positive results on the analog of Amitsur's conjecture for polynomial $T$-graded identities}

\begin{theorem}\label{TheoremTCancelAmitsur}
Let $A$ be a finite dimensional non-nilpotent $T$-graded associative algebra with $1$
over a field $F$ of characteristic $0$ for some cancellative semigroup $T$. Then
there exist constants $C_1, C_2 > 0$, $r_1, r_2 \in \mathbb R$, $d\in\mathbb N$,
such that $$C_1 n^{r_1} d^n \leqslant c^{T\text{-}\mathrm{gr}}_n(A) \leqslant C_2 n^{r_2} d^n\text{
for all }n \in \mathbb N.$$
\end{theorem}
\begin{corollary}
The graded analog of Amitsur's conjecture holds for such codimensions.
\end{corollary}
\begin{proof}[Proof of Theorem~\ref{TheoremTCancelAmitsur}.]
Note that graded codimensions do not change upon an extension of the base field.
The proof is analogous to the case of ordinary codimensions~\cite[Theorem~4.1.9]{ZaiGia}.
Hence we may assume $F$ to be algebraically closed.

By~\cite[Corollary 4.1]{KelarevBook}, $J(A)$ is a graded ideal.
By Proposition~\ref{PropositionTCancelWedderburn},
 $A/J(A)$ is the sum of graded ideals
that are $T$-graded-simple algebras. Now we apply Theorem~\ref{TheoremMainTGrAssoc}.
 \end{proof}

\begin{theorem}\label{TheoremTIdemAmitsur}
Let $A$ be a finite dimensional non-nilpotent $T$-graded associative algebra with $1$
over a field $F$ of characteristic $0$ for some left or right zero band $T$. Then
there exist constants $C_1, C_2 > 0$, $r_1, r_2 \in \mathbb R$,
such that $$C_1 n^{r_1} d^n \leqslant c^{T\text{-}\mathrm{gr}}_n(A) \leqslant C_2 n^{r_2} d^n\text{
for all }n \in \mathbb N$$
where $d=\PIexp(A)$ is the ordinary PI-exponent of $A$.
\end{theorem}
\begin{corollary}
The graded analog of Amitsur's conjecture holds for such codimensions.
\end{corollary}
\begin{proof}[Proof of Theorem~\ref{TheoremTIdemAmitsur}.]
Since $A$ is finite dimensional, we may assume that $T$ is finite. Again, without loss of generality,
we may assume $F$ to be algebraically closed.

By Propositions~\ref{PropositionTIdemGradedIdeals} and~\ref{PropositionTIdemWedderburn},
the Jacobson radical of $A$ is a graded ideal and $A/J(A)$ is the sum of graded ideals
that are simple algebras. Therefore, by Theorem~\ref{TheoremMainTGrAssoc} there exists
an integer $\PIexp^{T\text{-}\mathrm{gr}}(A)$. By Theorem~\ref{TheoremTIdemGradedWeddMalcev},
we can choose a graded maximal semisimple subalgebra $B$ such that $A=B\oplus J(A)$ (direct sum of graded subspaces). Then we may define an embedding $\varkappa \colon A/J \mathrel{\widetilde{\rightarrow}} B$ (see Theorem~\ref{TheoremMainTGrAssoc})
to be graded. Let $B=B_1 \oplus \ldots \oplus B_s$ (direct sum of ideals) for some simple algebras $B_i$.
Then by Proposition~\ref{PropositionTIdemGradedIdeals} the ideals $B_i$ are graded,
$A/J=\varkappa^{-1}(B_1) \oplus \ldots \oplus\varkappa^{-1}(B_s)$  (direct sum of graded ideals), and
\begin{equation*}\begin{split}\PIexp^{T\text{-}\mathrm{gr}}(A) = \max\dim\left( B_{i_1}\oplus B_{i_2} \oplus \ldots \oplus B_{i_r}
 \mathbin{\Bigl|}  r \geqslant 1,\right. \\ \left. ((FT)^*B_{i_1})A^+ \,((FT)^*B_{i_2}) A^+ \ldots ((FT)^*B_{i_{r-1}}) A^+\,((FT)^*B_{i_r})\ne 0\right)=\\ \max\dim\left( B_{i_1}\oplus B_{i_2} \oplus \ldots \oplus B_{i_r}
 \mathbin{\Bigl|}  r \geqslant 1,\ B_{i_1} A^+ \,B_{i_2} A^+ \ldots B_{i_{r-1}} A^+\,B_{i_r}\ne 0\right)=\\ \max\dim\left( B_{i_1}\oplus B_{i_2} \oplus \ldots \oplus B_{i_r}
 \mathbin{\Bigl|}  r \geqslant 1,\ B_{i_1} J(A) \,B_{i_2} J(A) \ldots B_{i_{r-1}} J(A)\,B_{i_r}\ne 0\right)=\PIexp(A)\end{split}\end{equation*}
 since $B_i$ are simple as ordinary algebras.
 \end{proof}
 \begin{remark}
 The equality $\PIexp^{T\text{-}\mathrm{gr}}(A)=\PIexp(A)$ does not imply the equality
 of codimensions. Indeed, let  $k\in\mathbb N$, $k\geqslant 2$, $A=M_k(F)$,
 $T=\lbrace t_1, \ldots, t_k\rbrace$ where $t_i t_\ell = t_\ell$
 for all $1\leqslant i,\ell \leqslant k$, $A^{(t_i)}=\langle e_{1 i},\ldots, e_{k i}\rangle_F$.
  Then $x_1^{(t_i)}$
 are linearly independent modulo $\Id^{T\text{-}\mathrm{gr}}(A)$.
 In order to check this, it is sufficient to substitute $x_1^{(t_i)} = e_{ii}$.
 Thus $c_1(A)=1 < c_1^{T\text{-}\mathrm{gr}}(A)=k$.
 \end{remark}

%************************************************************
%
%Radicals of Algebras Graded by Cancellative Linear Semigroups
%A. V. Kelarev
%Proceedings of the American Mathematical Society
%Vol. 124, No. 1 (Jan., 1996), pp. 61-65
%Published by: American Mathematical Society
%Article Stable URL: http://www.jstor.org/stable/2161398
%
%A. Dooms i Edmund Puczy\l lowski, On homogeneity of radicals of semigroup algebras. Semigroup Forum 68, s. 311-313, 2004.
%
%
%
%********************************
%    
%    Communications in Algebra
%Volume 29, Issue 11, 2001
%	
%SEMIGROUP GRADINGS OF FULL MATRIX RINGS
%
%
%S. D\u asc\ u alescu, A. V. Kelarev and L. van Wyk
%
%pages 5023-5031
%
%******************************************************************

\section*{Acknowledgements}

I am grateful to E.~Jespers, M.\,V.~Zaicev, and E.~Iwaki for helpful discussions.


\begin{thebibliography}{99}

 \bibitem{AljaGia} Aljadeff, E., Giambruno, A. Multialternating 
graded polynomials and growth of polynomial identities.
\textit{Proc. Amer. Math. Soc.} \textbf{141}:9 (2013), 3055--3065.

 \bibitem{AljaGiaLa} Aljadeff, E., Giambruno, A., La~Mattina, D.
 Graded polynomial identities and exponential growth.
\textit{J. reine angew. Math.}, \textbf{650} (2011), 83--100.


\bibitem{Bahturin} Bakhturin, Yu.\,A.
Identical relations in Lie algebras.
VNU Science Press, Utrecht, 1987.


  \bibitem{DrenKurs} Drensky, V.\,S.
        Free algebras and PI-algebras: graduate course in algebra.
        Singapore, Springer-Verlag, 2000.


 \bibitem{GiaLa} Giambruno, A., La~Mattina, D.
 Graded polynomial identities and codimensions: computing the exponential growth.
 \textit{Adv. Math.}, \textbf{225} (2010), 859--881.


\bibitem{ZaiGia} Giambruno, A., Zaicev, M.\,V.
Polynomial identities and asymptotic methods.
AMS Mathematical Surveys and Monographs Vol. 122,
 Providence, R.I., 2005.


 \bibitem{ASGordienko5}  Gordienko, A.\,S. Amitsur's conjecture for polynomial $H$-identities of $H$-module Lie algebras. \textit{Tran. Amer. Math. Soc.}
 \textbf{367}:1 (2015), 313--354.  

  
 \bibitem{ASGordienko8} Gordienko, A.\,S. Asymptotics of $H$-identities for associative algebras with an $H$-invariant radical. \textit{J. Algebra}, \textbf{393} (2013), 92--101.%\texttt{arXiv:1212.1321 [math.RA] 6 Dec 2012}
 
 
 \bibitem{ASGordienko9} Gordienko, A.\,S. Co-stability of radicals and its applications to PI-theory.
\textit{Algebra Colloqium} (to appear).

\bibitem{KelarevPI} Kelarev, A.\,V. On semigroup graded PI-algebras.
\textit{Semigroup Forum}, \textbf{47}:1 (1993), 294--298.

\bibitem{KelarevBook} Kelarev, A.\,V. Ring constructions and applications.
\textit{Series in Algebra} \textbf{9}, World Scientific, Singapore, 2002.

\bibitem{ZaiMishchFracPI} Mishchenko, S.\,P., Zaicev M.\,V.
 An example of a variety of Lie algebras
with a fractional exponent. \textit{J. Math. Sci. (New York)}, \textbf{93}:6 (1999),
977--982.

 
\bibitem{VerZaiMishch} Mishchenko, S.P., Verevkin, A.B., Zaitsev,
M.V.  A sufficient condition for coincidence of lower and upper
exponents of the variety of linear algebras.
\textit{Mosc. Univ. Math. Bull.},
\textbf{66}:2 (2011), 86--89.

\end{thebibliography}
\end{document}